\documentclass[3p,times]{elsarticle}	
\usepackage[english]{babel}	\usepackage[latin1]{inputenc}		\usepackage[T1]{fontenc}
\usepackage{amsmath,amsthm,amsfonts,amssymb,latexsym,graphics,wrapfig,enumerate}
\usepackage{dsfont,times,verbatim,mathrsfs,calrsfs,libertine,times,microtype}
\usepackage[all]{xy} 
\usepackage[d]{esvect}
\usepackage{epstopdf,auto-pst-pdf}
\usepackage{pstricks-add}
\usepackage{graphicx}
\usepackage{epstopdf}

\newtheorem*{lemma}{Lemma}
\newtheorem*{proposition}{Proposition}
\newtheorem*{thm}{Theorem}
\newtheorem*{corollary}{Corollary}

\theoremstyle{remark}
    
   \newtheorem*{example}{Example}

\theoremstyle{definition}
    \newtheorem*{definition}{Definition}
    \newcommand{\Nat}   {\mathbb{N}}            \newcommand{\Ent}   {\mathbb{Z}}
    \newcommand{\set}[1]{\left\{#1\right\}}     
    
    \newcommand{\cat}[1]{\mathscr{#1}}				\renewcommand{\cal}[1]{\cat{#1}}
    		\newcommand{\FM} {\cat{F}\!\!_M}
    \newcommand{\sm}[1] {\begin{smallmatrix}#1\end{smallmatrix}}
    \newcommand{\kk}    {\ensuremath{\mathbf{k}}}
    					\newcommand{\eps}   {\varepsilon}
    \newcommand{\To}    {\longrightarrow}
  \renewcommand{\mod}	{\operatorname{mod}}
    \newcommand{\Hom}   {\operatorname{Hom}}    \newcommand{\End}   {\operatorname{End}}
        \newcommand{\rad}   {\operatorname{rad}}
    \newcommand{\ind}   {\operatorname{ind}}    \newcommand{\add}   {\operatorname{add}}
    \newcommand{\Ker}   {\operatorname{Ker}}    

    \newcommand{\Ima}   {\operatorname{Im}}

    \newcommand{\Ext}   {\operatorname{Ext}}    
    \newcommand{\D}     {\operatorname{D}}		
    \newcommand{\Tr}    {{\rm Tr}\,}				\newcommand{\Ann}   {\operatorname{Ann}}
    \newcommand{\projd} {{\rm pd}\,}            \newcommand{\injd}  {{\rm id}\,}
    \newcommand{\gldim} {\operatorname{gl.dim.}}\newcommand{\repdim}{\operatorname{rep.dim.}}
  \renewcommand{\le}		{\leqslant}				 \renewcommand{\ge}		{\geqslant}

    \renewcommand{\vec}[1]   {\protect\vv{#1}}   
    \newcommand{\ivec}[1]   {\text{\reflectbox{\ensuremath{\protect\vv{\text{\reflectbox{$#1$}}}}}}}   

    \newcommand{\lra}{\To}

    \makeatletter
        \def\theequation{\thesection.\@arabic\c@equation}
        
        \def\dotsd{\,\mbox{}\mathinner{\mkern1mu\raise17\p@
            \vbox{\kern13\p@\hbox{.}}\mkern2mu
            \raise14\p@\hbox{.}\mkern2mu\raise11\p@\hbox{.}\mkern1mu}\mbox{}\,}
        \def\dotst{\,\mathinner{\mkern1mu\raise11\p@
            \vbox{\kern13\p@\hbox{.}}\mkern2mu
            \raise14\p@\hbox{.}\mkern2mu\raise17\p@\hbox{.}\mkern1mu}\,}
    \makeatother

\begin{document}

\title{The Representation Dimension of a Selfinjective Algebra of Euclidean Type}
\author{Ibrahim Assem}
    \address{D\'epartement de math\'ematiques, Facult\'e des sciences, Universit\'e de Sherbrooke,
            Sherbrooke, Qu\'ebec J1K 2R1, Canada.}
\author{Andrzej Skowro\'nski}
    \address{Faculty of Mathematics and Computer Science, Nicolaus Copernicus University,
           Chopina, 12/18 87-100 Toru\'n, Poland.}
\author{Sonia Trepode}
    \address{Departamento de Matem\'atica, Facultad de Ciencias Exactas y Naturales, Funes 3350,
            Universidad Nacional de Mar del Plata, 7600 Mar del Plata, Argentina.}


\begin{abstract}
    We prove that the representation dimension of a selfinjective algebra of
    euclidean type is equal to three, and give an explicit construction of the Auslander generator of its module category.
\end{abstract}

\maketitle

\section{Introduction}

The homological dimensions are useful algebraic invariants,
measuring how much an algebra or a module deviates from a situation considered to be ideal.
The representation dimension of an Artin algebra was introduced by Auslander in the early seventies~\cite{A}.
It measures the least global dimension of the endomorphism rings of modules which are at the same time
	generators and cogenerators of the module category,
   thus expressing the complexity of the morphisms in this category.
Part of the interest in this invariant comes from its relation with the finitistic dimension conjecture:
   it was proved by Igusa and Todorov that, if the representation dimension is at most three,
   then its finitistic dimension is finite~\cite{IT}.
Iyama has shown that the representation dimension of any algebra is finite~\cite{I}
   and Rouquier has shown that, for any positive integer $n$,
      there exists an algebra having $n$ as representation dimension~\cite{Ro}.
There were several attempts to understand this invariant and to compute it for classes of algebras,
   see for instance~\cite{APT,ACW,CHU,D,O,Ri1}.
In particular, it was shown in~\cite{CP} that
   the representation dimension of the trivial extension of an (iterated) tilted algebra equals three,
   and in~\cite{EHIS} that
   the representation dimension of a special biserial algebra is at most three.

In the present paper, we are interested in tame algebras.
It was shown by Auslander that an algebra is representation-finite if and only if
   its representation dimension equals two~\cite{A}.
Because Auslander's expectation was that this invariant would measure
   how far an algebra is from being representation-finite,
   it is natural to conjecture that the representation dimension of a tame algebra is at most three.

Among the best known and most studied classes of tame algebras are the tame selfinjective algebras, see~\cite{S2}.
In fact, it is shown in~\cite{BHS} that the selfinjective algebras
   socle equivalent to weakly symmetric algebras of euclidean type have representation dimension three.
Our objective here is to determine the representation dimension of
   selfinjective algebras of euclidean type (over an algebraically closed field).
We recall that, if $\vec\Delta$ is an euclidean quiver,
   then an algebra $A$ is called selfinjective of type $\vec\Delta$ whenever
   there exists a tilted algebra $B$ of type $\vec\Delta$
   and an infinite cyclic admissible group $G$ of automorphisms of the repetitive category $\hat B$ of $B$ such that $A\cong\hat B/G$.
Because of the main result of~\cite{S1}, this class of algebras coincides with
   the class of representation infinite domestic selfinjective algebras which admit simply connected Galois coverings,
   in the sense of~\cite{AS2}.
We prove the following theorem.
\begin{thm}
    Let $A$ be a selfinjective algebra of euclidean type.
    Then the representation dimension of $A$ is equal to three.
\end{thm}

Our strategy is the following.
We start by considering a class of algebras which we call domestic quasitube algebras,
   and prove that their representation dimension equals three.
These algebras are considered as building blocks for the
   repetitive category of a tilted algebra of euclidean type,
   which we obtain from them (and from the tilted algebra) by successive gluings.
Therefore, we show how to glue domestic quasitube algebras
   between them or to tilted algebras of euclidean type
   and prove that this does not change the representation dimension.
Finally, using Galois coverings, we prove our main theorem.
At each step, our proof is constructive: we give explicitly the generator-cogenerator
   of the module category for which the representation dimension is attained.

The paper is organised as follows.
After an introductory section devoted to fixing the notation
   and recalling useful facts about the representation dimension,
   section 3 is devoted to domestic quasitube algebras and
   section 4 to the gluings of such algebras.
We next recall in section 5 the necessary facts about the selfinjective algebras of euclidean type,
then prove our main theorem in section 6.

\section{The Representation Dimension}

\subsection{Notation}

Throughout this paper, $\kk$ denotes an algebraically closed field.
By algebra $A$ is meant a basic, connected, associative finite dimensional $\kk$-algebra with an identity.
Thus, there exists a connected bound quiver $(Q_A,I)$ and an isomorphism $A \cong \kk Q_A/ I$.
Equivalently, $A$ may be considered as a $\kk$-category with object class $A_0$ the set of points in $Q_A$,
and with set of morphisms $A(x,y)$ from $x$ to $y$ the quotient of the $\kk$-vector space $\kk Q_A(x,y)$ of linear combinations of paths in $Q_A$ from $x$ to $y$ by $I(x,y) = I \cap \kk Q_A (x,y)$, see~\cite{BG}.
A full subcategory $C$ of $A$ is \emph{convex} if,
    for each $\xymatrix@1@C=12pt{x_0 \ar[r]& x_1\ar[r]& \cdots \ar[r]& x_n}$
    in $A$ with $x_0, x_n \in C_0$,
    we have $x_i \in C_0$ for each $i$.
The algebra $A$ is \emph{triangular} if $Q_A$ is acyclic.

Here, $A$-modules will mean finitely generated right $A$-modules.
We denote by $\mod A$ the category of $A$-modules and
   by $\ind A$ a full subcategory consisting of a complete set of representatives of
   the isomorphism classes (isoclasses) of indecomposable $A$-modules.
For a point $x$ in  $Q_A$, we denote by $P(x)$ (or $I(x)$, or $S(x)$) the indecomposable projective
(or injective, or simple, respectively) $A$-module corresponding to $x$.
The projective (or injective) dimension of a module $M$ will be denoted by $\projd M$ (or $\injd M$, respectively)
and the global dimension of $A$ by $\gldim A$.
For a module $M$, the notation $\add M$ stands for the additive full subcategory of $\mod A$ with objects the direct sums of direct summands of $M$.
For two full subcategories $\cal{C}, \cal{D}$ of $\ind A$, the notation $\Hom_A(\cal{C}, \cal{D}) = 0$ means $\Hom_A(M,N) = 0$ for all $M$ in $\cal{C}$, $N$ in $\cal{D}$.
We then denote by $\cal D \vee \cal C$ the full subcategory of $\ind A$ having as objects
those of $\cal C_0 \cup \cal D_0$.
Finally, we denote by $\D = \Hom_{\kk}(-,\kk)$ the usual duality between $\mod A$ and $\mod A^{\rm{op}}$.

A \emph{path} in $\ind A$ from $M$ to $N$ is a sequence of non-zero morphisms
   \[\xymatrix@C=18pt{M = M_0 \ar[r]& M_1 \ar[r]&	\cdots	\ar[r]& M_t = N} \tag{*}\]
with all $M_i$ indecomposable.
We then say that $M$ is a \emph{predecessor} of $N$, or that $N$ is a \emph{successor} of $M$.

We use freely properties of the Auslander-Reiten translations $\tau_A = \D\Tr$ and $\tau^{-1}_A = \Tr\D$ and the Auslander-Reiten quiver $\Gamma(\mod A)$ of $A$ for which we refer to \cite{ARS,ASS}.
We identify points in $\Gamma(\mod A)$ with the corresponding $A$-modules and (parts of) components in $\Gamma(\mod A)$ with the corresponding full subcategories of $\ind A$.
For tubes, tubular extensions and coextensions, we refer the reader to~\cite{Ri},
and for tame algebras we refer to~\cite{S0,SS1,SS2}.

\subsection{Representation dimension}

The notion of representation dimension was introduced in~\cite{A} to which we refer for the original definition.
Hence, we use as definition the following characterisation, also from~\cite{A}.

\begin{definition}
    Let $A$ be a non-semisimple algebra. Its \emph{representation dimension} $\repdim A$ is the infimum of
    the global dimensions of the algebras $\End M$, where the module $M$ is at the same time a generator
    and a cogenerator of $\mod A$.
\end{definition}

Note that, if $M$ is a generator and a cogenerator of $\mod A$,
then it can be written as $M=A\oplus \D A\oplus M'$, where $M'$ is an $A$-module.
If $M$ is such a module and moreover $\repdim A =\gldim\End M$,
   then $M$ is called an \emph{Auslander generator} of $\mod A$.

For studying the representation dimension, it is convenient to use a functorial language.
A contravariant functor $F:(\add M)^{\rm op}\To Ab$ is called \emph{finitely presented}, or \emph{coherent}, if
   there exists a morphism $f : M_1 \to M_0$, with $M_1$, $M_0$ in $\add M$, which induces an exact sequence of functors
    \[\xymatrix{\Hom_A(-,M_1)\ar[rr]^{\Hom_A(-,f)}&& \Hom_A(-,M_0)\ar[r]& F\ar[r]&0.}\]
It is shown in~\cite{A} that the category $\FM$ of finitely presented functors from $(\add M)^{\rm op}$ to $Ab$
   is equivalent to $\mod \End M$ and, in particular, is abelian.

In this paper, we are particularly interested in algebras of representation dimension 3.
We recall that an algebra $A$ is representation-finite if and only if $\repdim A=2$, see~\cite{A}.
Therefore, if $A$ is representation-infinite, then $\repdim A  \ge 3$.
We have the following well-known characterisation of algebras with representation dimension 3, see~\cite{A,CP,EHIS,Xi1}.

\begin{lemma}
    Let $M$ be an $A$-module which is a generator and a cogenerator of $\mod A$.
    Then $\gldim(\End M)\le 3$ if and only if, for each $A$-module $X$,
    there exists a short exact sequence \[\xymatrix{ 0 \ar[r]& M_1 \ar[r]& M_0 \ar[r]& X \ar[r]& 0}\]
    with $M_0$, $M_1 \in \add M$, such that the induced sequence of functors
        \[\xymatrix{0\ar[r]& \Hom_A(-,M_1) \ar[r]& \Hom_A(-,M_0)\ar[r]& \Hom_A(-,X) \ar[r]& 0}\]
    is exact in $\FM$.
    In this case, $\repdim A \le 3$.\qed
\end{lemma}

In the situation of the lemma, not only does $M$ generate $X$, but also $\Hom_A(-,M)$ generates $\Hom_A(-,X)$.
This leads to consider the case where the morphism $\Hom_A(-,M_0) \to \Hom_A(-,X)$ is a projective cover.
It was proved in \cite{APT}(1.4) that if $X$ is generated by $M$,
   then there exists an epimorphism $f_0: M_0 \to X$, where $M_0 \in \add M$,
   such that $\Hom_A(-,f_0): \Hom_A(-,M_0) \to \Hom(-,X)$ is a projective cover.

Accordingly, a short exact sequence $\xymatrix@1@C=12pt{0 \ar[r]& M_1 \ar[r]& M_0 \ar[r]^{f_0}& X \ar[r]& 0}$
    such that $\Hom_A(-,f_0): \Hom_A(_,M_0) \to \Hom(-,X)$
    is a projective cover, is called an \emph{$\FM$-resolution} of $X$.

In this terminology, the previous lemma says that, if $M$ is a generator-cogenerator of $\mod A$,
then $\gldim(\End M)\le 3$ if and only if each $A$-module admits an $\FM$-resolution.

\subsection{Approximations}

An equivalent language is useful.
Let $M$ be any $A$-module. Given an $A$-module $X$,
a morphism $f_0: M_0 \lra X$  with $M_0 \in  \operatorname{add }M$ is an \emph{$\add M$-approximation} if,
   for any morphism $f_1: M_1 \to X$ there exists $g: M_1 \to M_0$ such that $f_1 = f_0g$:
   \[\xymatrix@R=6pt{
                     M_1\ar@{.>}[dd]_g
                        \ar[drr]^{f_1}	&&	  \\
                                       &&	X~.\\
                     M_0\ar[urr]^{f_0} &&   }
   \]

Equivalently, $f_0: M_0 \To X$  is an $\add M$-approximation if and only if
   $\Hom_A(-,f_0): \Hom_A(-,M_0) \To \Hom(-,X)$ is surjective in $\add M$.

Note that, if $X$ is generated by $M$, then any $\add M$-approximation $f_0: M_0 \to X$ of $X$ is surjective.
Indeed, let $f_1 : M_1 \to X$ with $M_1\in \add M$ be surjective.
Then there exists $g : M_1 \to M_0$ such that $f_1=f_0g$.
The surjectivity of $f_1$ implies that of $f_0$.

An $\add M$-approximation is \emph{(right) minimal} if each morphism $g: M_0 \to M_0$ such that $f_0g= f_0$ is an isomorphism.
Because of \cite{ARS}(I.2.1), if there exists an $\add M$-approximation,
   then there exists an $\add M$-approximation which is minimal and is then called a \emph{minimal $\add M$-approximation}.

A short exact sequence
    \[\xymatrix{ 0\ar[r]	&	M_1\ar[r]	&	M_0\ar[r]^{f_0}	&	X\ar[r]	&	0	}\]
with $M_1$, $M_0\in\add M$ is an \emph{$\add M$-approximating sequence} if $f_0:M_0\to X$ is an $\add M$-approximation of $X$.
It is a \emph{minimal $\add M$-approximating sequence} if moreover $f_0$ is minimal.
We need~\cite{ACW}(1.7) which we reprove here because it is central to our considerations.

\begin{lemma}
    Let $M, X$ be $A$-modules. If there exists an $\add M$-approximating sequence of $X$,
    then there exists a minimal $\add M$-approximating sequence which is moreover
    a direct summand of any $\add M$-approximating sequence.
\end{lemma}

\begin{proof}
    Let $M$, $X$ be $A$-modules and $\xymatrix@1@C=25pt{0 \ar[r]&M_1 \ar[r]& M_0 \ar[r]^{f_0}& X \ar[r]& 0}$ be an $\add M$-approximating sequence. Because of~\cite{ARS}(I.2.1), there exists a minimal $\add M$-approximation $f_0':M_0'\to X$.
    Also, $f_0'$ is surjective as observed above.
    Let $M_1'=\Ker f_0'$. Because each of $f_0$, $f_0'$ is an $\add M$-approximation of $M$,
    we get a commutative diagram with exact rows
    \[\xymatrix{
       0\ar[r]&	M_1' \ar[r]\ar[d]_u 	 & M_0'\ar[r]^{f_0'}\ar[d]_v   & X\ar@{=}[d]\ar[r]&0	\\
       0\ar[r]&	M_1 \ar[r]\ar[d]_{u'} & M_0 \ar[r]^{f_0}\ar[d]_{v'} & X\ar@{=}[d]\ar[r]&0	\\
       0\ar[r]&	M_1' \ar[r]			 	 & M_0'\ar[r]^{f_0'}			    & X 				\ar[r]&0~.   }\]
     Because $f_0'$ is minimal, $v'v$ is an isomorphism. Hence, so is $u'u$.
     Therefore, $u$ and $v$ are sections.
\end{proof}

\subsection{Approximating sequences and {$\FM$}-resolutions}

We now prove that these two terminologies are equivalent.

\begin{lemma}
   Let $M,X$ be $A$-modules and $\xymatrix@1@C=15pt{0 \ar[r]&M_1 \ar[r]& M_0 \ar[r]^{f_0}& X \ar[r]& 0}$ be
   exact with $M_1$, $M_0\in\add M$.
   Then this sequence is an $\FM$-resolution if and only if it is a minimal $\add M$-approximating sequence.
\end{lemma}

\begin{proof}
   It suffices to prove that an $\add M$-approximation $f_0:M_0\to X$ is minimal if and only if
      $\Hom_A(-,f_0):\Hom_A(-,M_0)\to \Hom_A(-,X)$ is a projective cover in $\FM$.

   Indeed, $\Hom_A(-,f_0)$ is a projective cover if and only if, for any epimorphism $\Hom_A(-,f'):\Hom_A(-,M')\to\Hom_A(-,X)$
   with $M'\in\add M$, there exists a retraction $\Hom_A(-,g):\Hom_A(-,M')\to\Hom_A(-,M_0)$ such that $\Hom_A(-,f_0)\Hom_A(-,g)=\Hom_A(-,f')$.
   This is equivalent to saying that there exists a retraction $g:M'\to M_0$ such that $f_0g=f'$, or to saying that among all $\add M$-approximations of $X$, $f_0$ is the one whose domain $M_0$ has least length. Because of~\cite{ARS}(I.2.2), this is the same as requiring that $f_0$ is minimal.
\end{proof}

\subsection{Tilted Algebras}

We recall the definition of a tilted algebra.
Let $A$ be an algebra.
A module $T_A$ is said to be a \emph{tilting module} if $\projd T \leq 1$,  $\Ext^1_A(T,T)= 0$
and there exists a short exact sequence
    $\xymatrix@1@C=12pt{0 \ar[r]& A_A \ar[r]& T'_A \ar[r]& T''_A \ar[r]& 0}$ with $T'_A$, $T''_A$ in $\add T$.
Let $H$ be a hereditary algebra.
An algebra $A$ is \emph{tilted of type $H$} if there exists a tilting $H$-module $T$
   such that $A = \End T_H$.
If $H$ is the path algebra of a Dynkin (or euclidean, or wild) quiver $Q$, then we say that $A$ is tilted of Dynkin (or euclidean, or wild, respectively) type $Q$. For tilting theory we refer the reader to~\cite{ASS}.

Tilted algebras are characterised by the existence of slices in their module categories.
We recall from~\cite{Ri2}(Appendix) that a class $\Sigma$ in $\ind A$ is called a \emph{complete slice} if:
\begin{enumerate}[\indent(1)]
    \item $\Sigma$ is sincere, that is, if $P$ is any projective $A$-module,
        then there exists $U \in \Sigma$ such that $\Hom_A(P,U) \not= 0$;
    \item $\Sigma$ is convex, that is, if $\xymatrix@1@C=12pt{U_0 \ar[r]& U_1 \ar[r]&\cdots \ar[r]&U_t}$ is a path
        in $\ind A$ with $U_0$, $U_t\in\Sigma$ then $U_i \in \Sigma$ for all $i$;
    \item if $\xymatrix@1@C=12pt{0 \ar[r]& U \ar[r]& V \ar[r]& W \ar[r]& 0}$ is
        an almost split sequence, then at most one of $U$, $W$ lies in $\Sigma$.
        Moreover, if an indecomposable summand of $V$ lies in $\Sigma$, then $U$ or $W$ lies in $\Sigma$.
\end{enumerate}

It is shown (see, for instance, \cite{Ri}) that an algebra is tilted if and only if it contains a complete slice $\Sigma$.
In this case, the module $T = \bigoplus_{U \in \Sigma} U$ is a tilting module called the \emph{slice module} of $\Sigma$.
We need the following fact from~\cite{APT}.

\begin{lemma}
    Let $A$ be a tilted algebra, and $T$ the slice module of a complete slice.
    Then
    \begin{enumerate}[\indent(a)]
        \item for any $A$-module $X$ generated by $T$, there exists a minimal
            $\add(T \oplus \D A)$-approximating sequence for $X$ of the form
                \[\xymatrix@1{0 \ar[r]& T_1 \ar[r]& T_0 \oplus I_0 \ar[r]& X \ar[r]& 0}\]
            with $T_1$, $T_0 \in \add T$ and $I_0$ injective;
        \item the module $M = A \oplus \D A \oplus T$ is an Auslander generator for $\mod A$ and $\repdim A\le3$.
    \end{enumerate}
\end{lemma}

\begin{proof}
   \begin{enumerate}[\indent(a)]
      \item Existence of a minimal $\add(T\oplus\D A)$-approximation follows from~\cite{APT}(1.4) and~\cite{ARS}(I.2.2).
      That the kernel of this approximation lies in $\add T$ follows from~\cite{APT}(2.2)(f).
      \item This is~\cite{APT}(2.3).
   \end{enumerate}
\end{proof}

\section{Domestic Quasitube Algebras}

\subsection{The definition}

In this section, we define domestic quasitube algebras and prove that their representation dimension is 3.
First, we recall that a family of pairwise orthogonal generalised standard components $\cat C = \left( \cat C_i\right)_{i \in I}$ in the Auslander-Reiten quiver of an algebra $A$
is called a \emph{separating family} of components if the indecomposables not in $\cat C$ split into two classes $\cat P$ and $\cat Q$ such that:
    \begin{enumerate}[\indent(a)]
        \item $\Hom_A(\cat Q,\cat P)=0$, $\;\Hom_A(\cat Q,\cat C)=0\;$ and
              $\;\Hom_A(\cat C,\cat P)=0$, and
        \item any morphism from ${\cal P}$ to ${\cal Q}$ factors through add $C$.
    \end{enumerate}

We thus have $\ind A = \cat P \vee \cat C \vee \cat Q$.

For the admissible operations, we refer to~\cite{AST} or~\cite{MS1}.


Given a tame concealed algebra $C$, an algebra $A$ is a \emph{quasitube enlargement} of $C$ if $A$ is obtained from $C$ by an iteration of the admissible operations ad$\,$1), ad$\,$1*), ad$\,$2), ad$\,$2*) either on a stable tube of $\Gamma(\mod C)$ or on a quasitube obtained from a stable tube by means of the operations done so far.
A quasitube enlargement $A$ of $C$ is a \emph{domestic quasitube enlargement} provided $A$ is a domestic algebra.

\begin{definition}
   An algebra $A$ is a \emph{domestic quasitube algebra} provided
      $A$ is a domestic quasitube enlargement of a tame concealed algebra such that
         all projectives, and all injectives in its quasitubes are projective-injective.
\end{definition}

Thus, a quasitube in a domestic quasitube algebra becomes a stable tube after deletion of
   the projective-injectives and all arrows incident to them.

From now on, and until the end of the paper, we use the term quasitube algebra in this particular restricted sense.

Specialising theorems 3.5, 4.1 and Corollary 4.2 of \cite{AST} to our situation,
   we get the following structure theorem for domestic quasitube algebras and their module categories.

\begin{thm}
    Let $A$ be a domestic quasitube algebra obtained as a quasitube enlargement of a tame concealed algebra $C$.
    Then:
    \begin{enumerate}[\indent(a)]
        \item $A$ has a sincere separating family $\cal C$ of quasitubes obtained from the stable
            tubes of $C$ by the corresponding sequence of admissible operations;
        \item there is a unique maximal branch coextension $A^{-}$ of $A$ which is a full convex subcategory
            of $A$, and which is a tilted algebra of euclidean type;
        \item there is a unique maximal branch extension $A^+$ of $A$ which is a full convex subcategory
            of $A$, and which is a tilted algebra of euclidean type;
        \item $\ind A = \cal P \vee \cal C \vee \cal Q$, where $\cal P$ is the
            postprojective component of $\Gamma(\mod A^{-})$ while $\cal Q$ is the preinjective
                component of $\Gamma(\mod A^+)$.
    \end{enumerate}
\end{thm}

We may observe that an algebra $A$ is a domestic quasitube algebra if and only if it is a domestic algebra
with a separating family of quasitubes:
this indeed follows easily from Theorems A and F of~\cite{MS2}.

It follows easily from the description of the theorem that $A^-$ is the left support algebra of $A$, while $A^+$ is its right support algebra, in the sense of~\cite{ACT}. We recall that support algebras of subcategories have been used in~\cite{APT, ACW} in order to calculate the representation dimension.
It is then natural to use them in our context.

\begin{example}
   Let $A$ be given by the quiver
   \[\xymatrix{	
         1	&	2 \ar[l]_\delta	&		& 3\ar@<-.5ex>[ll]_\beta
                                          \ar@<.5ex>[ll]^\gamma	&	4 \ar[l]_\alpha
                                                               \ar[dll]^\lambda \\
            &		&  5\ar[ull]^\mu		         }\]
   bound by $\alpha\beta=0$, $\beta\delta=0$, $\alpha\gamma\delta=\lambda\mu$.
   In this example $A^-$ and $A^+$ are the full subcategories of $A$ with object classes
   $A_0^-=A_0\setminus \set{4}$ and $A_0^+=A_0\setminus\set{1}$, respectively.
   Both are tilted of type $\tilde{\mathbb{A}}$.
   The Auslander-Reiten quiver $\Gamma(\mod A)$ has the form
    \[\xymatrix@R=5pt@C=4pt{
       &&&&&&&&\ar@{-}@/^4pt/[]+0;[r]+0
               \ar@{-}@/_4pt/[]+0;[r]+0 \ar@{-}[ddddd];[]+0&\ar@{-}[ddddd];[]+0
                &\ar@{-}@/^4pt/[]+0;[r]+0
                 \ar@{-}@/_4pt/[]+0;[r]+0\ar@{-}[ddddd];[]+0&\ar@{-}[ddddd];[]+0
                  & {\sm{3\\2}}
                     \ar[dr]
                     \ar@{--}[ddddd]
                      &   &{\sm{3\\2\\1}}
                            \ar[dr]
                            &  & 5\ar[dr]
                                     & & {\sm{4\\3\\2}}
                                           \ar[dr]
                                              &&{\sm{3\\2}}
                                                 \ar@{--}[ddddd]&&&& \\
       &{\sm{2\\1}}
          \ar[dr]
          \ar@(r,u)[dddrrr]
          &&&  {\sm{~3\\522\\1~}}
                \ar[dr]
                \ar@(r,u)[ddrr]
              &&&\cdots&&&&&&{\sm{33\\22\\1}}
                                \ar[ur]
                                \ar[dr]
                                &&{\sm{3\\25\\1}}
                                      \ar[ur]
                                      \ar[]+<5pt,0pt>;[r]-<5pt,0pt>
                                      \ar[dr]
                                   &{\sm{4\\35\\2~\\1}}
                                      \ar[]+<5pt,0pt>;[r]-<5pt,0pt>
                                      &{\sm{4\\35\\2~}}
                                         \ar[ur]
                                            &&{\sm{4\\33\\22}}
                                                  \ar[ur]
                                                  &&\cdots
                                                     &{\sm{4\\335\\2}}
                                                        \ar[drr]
                                                        \ar@/^/[ddd]
                                                        &&&{\sm{4\\3}}\ar[dr]\\
       1\ar[ur]
        \ar[dr]
        && {\sm{25\\1}}
              \ar[urr]
              \ar[dr]
                 &&&{\sm{3\\22}}
                       \ar[dr]
                       &&&&&&&&&{\sm{33\\225\\1}}
                               \ar[ur]
                               &&{\sm{3\\2}}
                                  \ar[ur]&&&&&&&&{\sm{4\\35}}\ar[ur]\ar[dr]&&4			\\
       &{\sm{5\\1}}
          \ar[ur]
          &&2\ar[urr]
             \ar[dr]
             &&&\cdot&&&&&&&&&&&&&&&&&3\ar[ur]&&{\sm{4\\5}}\ar[ur]\\
       &&&&{\sm{3\\22\\1~}}
             \ar[urr]
             \ar@/_/[uuu]
             &&&\cdots
                &&&&&&&&&&&&&&\cdots&{\sm{4~\\33\\2}}\ar[ur]\ar@(r,d)[uuurrr]\\
       &&&&&&&&&&&&&&&&&&&&&&&&&&}
      \]
   where the indecomposables are represented by their Loewy series.
   The shown quasitube is obtained by identifying along the vertical dotted lines.
\end{example}

Some additional remarks are in order.
Let $A$ be a domestic quasitube algebra.
In the notation of the theorem, the postprojective component $\cat P$ of $\Gamma(\mod A)$
   coincides with the postprojective component of $\Gamma(\mod A^{-})$.
Now, we know that $A^{-}$ is a branch coextension of a tame concealed algebra, and is also a tilted algebra of euclidean type.
Therefore, $\cat P$ contains a complete slice $\Sigma^{-}$ of mod $A^{-}$.
However, $\Sigma^{-}$ is clearly not a complete slice in mod $A$, because the quasitubes of the family ${\cal C}$ generally contain projective-injectives.
It is a right section in the sense of~\cite{A1}.
We recall the definition.
Let $\Gamma$ be a translation quiver.
A full subquiver $\Sigma$ of $\Gamma$ is called a \emph{right section} if
    \begin{enumerate}[\indent(1)]
        \item $\Sigma$ is acyclic;
        \item for any $x \in \Gamma_0$ such that there exist $y\in\Sigma_0$ and a
            path from $y$ to $x$ in $\Gamma$, there exists a unique $n\ge0$ such
            that $\tau^n x \in \Sigma_0$;
        \item $\Sigma$ is convex in $\Gamma$.
    \end{enumerate}

Dually, one defines \emph{left sections}.
A subquiver which is at the same time a right and a left section is called a \emph{section}, see \cite{ASS}.

It follows easily from its definition that $\Sigma^{-}$ is a right section
in the postprojective component $\cal P$ of $\Gamma(\mod A)$
and that $A^{-} = A / \Ann\Sigma^{-}$, where $\Ann\Sigma^{-} = \bigcap_{U\in\Sigma^{-}} \Ann U$.

Dually, the preinjective component $\cat Q$ of $\Gamma( \mod A)$ contains
a complete slice $\Sigma^+$  in $\mod A^+$, which is not a slice in $\mod A$, but rather a left section.
Moreover, $A^+ = A / \Ann\Sigma^+$.

\subsection{Restriction of injectives}

When dealing with a module over a domestic quasitube algebra,
we need to consider its restriction to $A^{-}$, that is, the largest $A^{-}$-submodule of this module.
We recall from theorem 3.1 that each of $A^{-}$ and $A^+$ is a full convex subcategory of $A$,
   with $A^{-}$ closed under successors and $A^+$ closed under predecessors.
The following lemma is useful.

\begin{lemma}
    Let $A$ be an algebra, $B$ a quotient of $A$, $I$ an indecomposable injective $A$-module having socle in $\mod B$.
    Then the largest $B$-submodule $I'$ of $I$ is an indecomposable injective $B$-module.
\end{lemma}
\begin{proof}
   Let $f:X\to Y$ be a monomorphism and $g:X\to I'$ be a morphism in $\mod B$.
   Let also $j:I'\to I$ denote the canonical inclusion.
   Then we have a diagram as shown in $\mod A$.
   Because $I$ is injective in $\mod A$,
   there exists a morphism $h:Y\to I$ such that $hf=jg$.
   Because $Y$ is a $B$-module, $h(Y)\subseteq I'$.
   That is, there exists $h':Y\to I'$ such that $h'f=g$.
   \[\xymatrix@C=48pt{0\ar[r]&X\ar[r]^f
                        \ar[d]_g 		& Y\ar@{.>}[dl]^{h'}
                                          \ar@{.>}[ddl]^{h}\\
                      & I'\ar[d]_j\\
                      & I	                     }\]
\end{proof}

\subsection{Constructing the approximating sequence.}

Let $A$ be a domestic quasitube algebra.
As seen before, there exist a right section $\Sigma^{-}$ in the postprojective component of $\Gamma(\mod A)$ and a left section $\Sigma^+$ in the preinjective component.
Moreover, $\Sigma^{-}$ and $\Sigma^+$ are respectively complete slices in mod $A^{-}$ and mod $A^+$.
Let $T^{-}$ and $T^+$  denote the respective slice modules of $\Sigma^{-}$ and $\Sigma^+$.
We may choose the slice $\Sigma^+$ in such a way that every preinjective indecomposable $A$-module with support lying completely in the extensions branches of $A^+$ is a successor of $\Sigma^+$.
This is possible because there are only finitely many such indecomposables.
As an easy consequence, the restriction to $A^{-}$ of any preinjective predecessor of $\Sigma^+$ is nonzero.
We then set
    $$M = A \oplus \D A \oplus T^{-} \oplus T^+ \oplus \D A^{-}$$

We prove, in theorem 3.4  below, that $M$ is an Auslander generator for mod $A$.

Also useful is the module $N = A^{-}\oplus T^{-} \oplus \D A^{-}$.
Indeed, we recall that, because of lemma 2.5, every indecomposable $A^{-}$-module admits a minimal add $N$-approximating sequence.

\begin{proposition}
    Let $A$ be a domestic quasitube algebra and $X$ be an indecomposable $A$-module
    whose restriction $Y$ to $A^{-}$ is nonzero.
    Let also
        \[\xymatrix@R=12pt{
           0 \ar[r]& L \ar[r]^-r    & N_0  \ar[r]^-q & Y \ar[r]& 0\\
           0 \ar[r]& L' \ar[r]^{i'} & P \ar[r]^{p'}& X/Y \ar[r]&0 }\]
    be respectively a minimal $\add N$-approximating sequence and a
    projective cover of $X/Y$ in $\mod A$.
    Then there exists an $A$-module $K$ such that we have exact sequences
      \[\xymatrix@R=12pt{
         0 \ar[r]& K \ar[r]^-s & N_0     \ar[r]^-t & X \ar[r]& 0\\
         0 \ar[r]& L \ar[r]^\ell & K \ar[r]^{\ell'}& L' \ar[r]&0. }\]
   Moreover, $K\cong L\oplus L'$.
   In particular, $K\in\add N$.
\end{proposition}
\begin{proof}
   Let $\xymatrix@1@C=15pt{0\ar[r]&Y\ar[r]^i&X\ar[r]^-p&X/Y\ar[r]&0}$ be exact.
   Because $P$ is projective, there exists a morphism $g:P\to X$ such that $pg=p'$.
   We now claim that $t=(iq,g):N_0\oplus P\to X$ is an epimorphism.
   Let $x\in X$. Because $p'$ is surjective, there exists $z\in P$ such that $p(x)=p'(z)=pg(z)$.
   Therefore, $x-g(z)\in\Ker p=\Ima i=\Ima iq$ because $q$ is surjective.
   This establishes our claim.

   Let $(K,s)$ denote the kernel of $t$.
   The snake lemma yields a commutative diagram with exact rows and columns:
     \[\xymatrix@C=40pt@R=20pt{
                &0\ar[d]		  &0\ar[d]			  &0\ar[d]	 &				\\
        0 \ar[r]& L \ar[r]^-r
                    \ar[d]^\ell & N_0 \ar[r]^-q
                                      \ar[d]^{k=\left(1\atop 0\right)}
                                                  & Y \ar[r]
                                                     \ar[d]^i& 0			\\
        0 \ar[r]& K \ar[r]^-s
                    \ar[d]^{\ell'} & N_0\oplus P
                                      \ar[r]^-{t=(iq,g)}
                                      \ar[d]^{k'=(0~1)}
                                                  & X \ar[r]
                                                     \ar[d]^p& 0			\\
        0 \ar[r]& L' \ar[r]^{i'}
                     \ar[d] 		& P \ar[r]^{p'}
                                    \ar[d]		  & X/Y \ar[r]
                                                      \ar[d] &0.			\\
                & 0				  & 0					  & 0 		 & }\]
   Now observe that $K$ is an $A^-$-module.
   Indeed, because of lemma 2.5, $L\in\add N$ and, in particular, is an $A^-$-module.
   On the other hand, $L'$ is the largest $A^-$-submodule of $P$.
   Because of lemma 3.2, $L'$ is actually an injective $A^-$-module.
   This establishes our claim.

   Because $A^-$ is a projective $A$-module, applying the exact functor $\Hom_A(A^-,-)$ to the middle row of
   the above diagram yields an exact sequence
   \[\xymatrix@C=40pt@R=20pt{ 0 \ar[r]& K \ar[r]^-s& N_0\oplus L'\ar[r]^{t'} & Y \ar[r]& 0 }\]
   in $\mod A^-$, where we have used that $K$ is an $A^-$-module, hence $s(K)\subseteq N_0\oplus L'$.
   We deduce a commutative diagram with exact rows
     \[\xymatrix@C=40pt@R=20pt{
        0 \ar[r]& L \ar[r]^r\ar[d]^\ell & N_0 \ar[r]^q \ar[d]^k & Y\ar@{=}[d] \ar[r]&0\\
        0 \ar[r]& K \ar[r]^-s & N_0 \oplus L' \ar[r]^-{t'} & Y \ar[r]& 0 }\]
   where we have used that $k(N_0)\subseteq N_0\oplus L'$.
   Because $t':N_0\oplus L'\to Y$ is a morphism from a module in $\add N$ to $Y$,
   while $q$ is a minimal $\add N$-approximation,
   there exists $k'':N_0\oplus L'\to N_0$ such that we have a commutative diagram with exact rows
	\[\xymatrix@C=40pt@R=20pt{
  	0 \ar[r] & L \ar[r]^-r
                \ar[d]^\ell 		   & N_0 \ar[r]^-q
                                      \ar[d]^{k}		 & Y \ar[r]
                                                           \ar@{=}[d] & 0   \\
	0 \ar[r]	& K \ar[r]^-s
                 \ar[d]^{\ell''} 	& N_0\oplus L'
                                         \ar[r]^-{t'}
                                         \ar[d]^{k''} & Y \ar[r]
                                                           \ar@{=}[d] & 0   \\
	0 \ar[r] & L \ar[r]^{r}				& N_0 \ar[r]^{q}	& Y \ar[r]		 & 0}\]
   where $\ell''$ is deduced by passing to the kernels.
   Minimality of $q$ yields that $k''k$ is an isomorphism.
   Therefore, so is $\ell''\ell$.
   In particular, $\ell$ is a section and the short exact sequence
    \[\xymatrix@C=40pt@R=20pt{ 0 \ar[r]& L \ar[r]^-\ell& K\ar[r]^{\ell'} & L' \ar[r]& 0 }\]
   splits, that is, $K\cong L\oplus L'$.
   Finally, $K\in\add N$ because $L\in\add N$, while $L'\in\add \D A^-\subseteq \add N$.
\end{proof}

\subsection{Representation dimension}
\begin{thm}
    Let $A$ be a domestic quasitube algebra. Then $\repdim A = 3$.
\end{thm}
\begin{proof}
   Because $A$ is representation-infinite, it suffices to show that $\repdim A\le 3$.
   Let $M$ be as in 3.3 above.

   Let $X$ be a postprojective $A$-module (thus postprojective $A^-$-module).
   Because of lemma 2.5, there exists a minimal $\add(A^-\oplus T^-\oplus \D A^-)$-approximating sequence
    \[\xymatrix@C=40pt@R=20pt{ 0 \ar[r]& M_1 \ar[r]& M_0\ar[r] & X \ar[r]& 0. }\]
   Because projective $A^-$-modules are also projective $A$-modules, we have $M_1$, $M_0\in\add M$.
   Let $f:M'\to X$ be a nonzero morphism, where we may assume, without loss of generality, that $M'$ is an
   indecomposable summand of $M$.
   If $X$ is a predecessor of $\Sigma^-$, then $M'$ is projective and trivially $f$ lifts to a morphism $M'\to M_0$.
   If $X$ is a successor of $\Sigma^-$, then $M'\in\add(A^-\oplus T^-\oplus \D A^-)$, hence $f$ factors through $T^-$.
   This shows that we have an $\add M$-approximating sequence.

   Let now $X$ be a nonpostprojective $A$-module whose restriction $Y$ to $A^-$ is nonzero.
   In particular, this is the case for all indecomposables which belong either to the separating family of
   quasitubes or to the predecessors of $\Sigma^+$ in the preinjective component
   (this is due to our special choice of the slice $\Sigma^+$ defining $T^+$).
   Because of proposition 3.3, there exists an exact sequence
           \[\xymatrix@1@R=12pt{
              0 \ar[r]& K \ar[r]^-s    & T_0\oplus I_0\oplus P  \ar[r]^-t & X \ar[r]& 0  }\]
    with $K$, $T_0\oplus I_0\oplus P\in\add(A^-\oplus T^- \oplus \D A^-)$.
    We show that it is an $\add M$-approximating sequence.
    Let $f:M'\to X$ be a nonzero morphism with $M'\in \add M$ indecomposable.
    Because $f\ne 0$, we have $M'\notin\add(T^+\oplus \D A^+)$.
    Therefore, $M'\in\add(A\oplus T^- \oplus \D A^-)$.
    If $M'$ is projective then $f$ trivially lifts to a morphism $M'\to T_0\oplus I_0\oplus P$.
    If $M'\in\add(T^-\oplus\D A^-)$, then $f(M')$ lies in the restriction $Y$ of $X$ to $A^-$.
    Therefore, $f$ lifts to a morphism $M'\to T_0\oplus I_0$ and consequently to a morphism $M'\to T_0\oplus I_0\oplus P$.

    Let finally $X$ be a successor of $T^+$.
    Because of lemma 2.5, there exists a minimal $\add(T^+\oplus \D A^+)$-approximating sequence
           \[\xymatrix@1@R=12pt{
              0 \ar[r]& T_1' \ar[r]   & T_1\oplus I_1 \ar[r] & X \ar[r]& 0 . }\]
    Because $A^+$-injectives are also $A$-injectives, we have $T_1'$, $T_1\oplus I_1 \in\add M$.
    Let $f:M'\to X$ be a nonzero morphism with $M'\in\add M$ indecomposable.
    If $M'\in\add(T^+\oplus \D A^+)$, then clearly $f$ lifts to a morphism $M'\to T_1\oplus I_1$.
    If $M'\notin \add(T^+\oplus \D A^+)$, then $f$ must factor through $T^+$ and thus also lifts to a
    morphism $M'\to T_1\oplus I_1$.
    This finishes the proof.
\end{proof}

\begin{example}
   In example 3.1, we may take $\Sigma^-=\set{\sm{2\\1},\sm{2\,5\\1},2,\sm{3\\2\,2\\1\,~}}$ and
   $\Sigma^+=\set{\sm{4\\3},\sm{4\\3\,5},3,\sm{~\,4\\3\,3\\2}}$.
   Indeed, the only indecomposables with support in the extension branch which are preinjective, namely $\sm{4\\5}$ and $4$,
   are successors of $\Sigma^+$.

   On the other hand, $\D A^-=\sm{3\,~\\2\,5\\1}\oplus \sm{3\,3\\2}\oplus 3\oplus 5$, so that the Auslander generator is
   \[ M= 1\oplus \sm{2\\1}\oplus \sm{5\\1}\oplus \sm{3\\2\,2\\1\,~}\oplus \sm{4\\{3\atop 2}\,5\\ 1}\oplus\sm{4\,~\\3\,3\\2}\oplus
   \sm{4\\3}\oplus\sm{4\\5}\oplus 4\oplus 3\oplus 5\oplus 2\oplus \sm{2\,5\\1}\oplus \sm{3\\2\,5\\1}\oplus\sm{3\,3\\2}\oplus \sm{4\\3\,5}.
   \]
\end{example}

\section{Gluings of Algebras}

\subsection{Finite gluings}

The purpose of this section is to show how to glue together algebras having representation dimension three
   in order to construct larger algebras having the same representation dimension.
We need to introduce a notation.
Let $A$ be a representation-infinite algebra, having a right section $\Sigma$ in a postprojective component,
   or a left section, also denoted by $\Sigma$, in a preinjective component of $\Gamma(\mod A)$.
We denote by $\ivec{\Sigma}$ the set of all indecomposable $A$-modules $X$ which are predecessors of $\Sigma$,
   that is, such that there exist $Y$ in $\Sigma$ and a path in $\mod A$ from $X$ to $Y$.
Dually, we denote by $\vec{\Sigma}$ the set of all indecomposable $A$-modules which are successors of $\Sigma$.

\begin{definition}
    We say that an algebra $A$ is a \emph{finite gluing} of two algebras $B$ and $C$, in symbols $A = B*C$, if
    \begin{itemize}
        \item [\indent(FG1)] $\Gamma(\mod B)$ has a unique preinjective component $\cat Q_B$ containing a
            left section $\Sigma^+_B$ and $\Gamma(\mod C)$ has a unique postprojective
            component $\cat P_C$ containing a right section $\Sigma^-_C$;
        \item [\indent(FG2)] $\Gamma(\mod A)$ has a separating component $\cat G$ such that:
            \begin{enumerate}[\indent(1)]
                \item $\cal G$ contains a left section isomorphic to $\Sigma^+_B$
                    and the indecomposable $A$-modules in $\cat G$ which precede it are exactly those
                    of $\ivec{\Sigma}_B^+ \cap \cat Q_B$,
                \item $\cal G$ contains a left section isomorphic to $\Sigma^-_C$
                    and the indecomposable $A$-modules in $\cat G$ which succede it are exactly those
                    of $\vec{\Sigma}^{-}_C \cap \cat P_C$,
                \item $(\ivec{\Sigma}^+_B \cap \cat Q_B) \cup
                        (\vec{\Sigma}^-_C \cap \cat P_C)$ is cofinite in $\cat G$;
            \end{enumerate}
        \item [\indent(FG3)] the remaining indecomposable $A$-modules belong to one of two classes:
            \begin{enumerate}[\indent(1)]
                \item those which precede $\cal G$ are the indecomposable $B$-modules
                        in $\ivec{\Sigma}^+_B \setminus \cat Q_B$,
                \item those which succede $\cal G$ are the indecomposable $C$-modules
                        in $\vec{\Sigma}^-_C \setminus \cat P_C$.
            \end{enumerate}
    \end{itemize}
\end{definition}

Thus we have
   \[\ind A = (\ivec{\Sigma}^+_B \setminus \cat Q_B)
                   \,\vee\; \cat G \;
                   \vee\,  (\vec{\Sigma}^-_C \setminus \cat P_C).\]
The component $\cat G$ is called the \emph{glued component} of $\Gamma(\mod A)$.
The latter may be visualised as
\[\xymatrix@C=10pt@R=10pt{
   \ar@{-}@(r,r)[dddd]
   	&&					&&					&&		&&			&&  \ar@{-}@(l,l)[dddd]\\
      &\ar@{.}[r]
       &\ar@{-}[r]
                     &&
                         				&& 	&\ar@{-}[r]&\ar@{.}[r]
                                                   &&  \\
   {\ivec{\Sigma}^+_B \setminus \cat Q_B}
      &&    {\ivec{\Sigma}^+_B \cap \cat Q_B}
                      && {\Sigma^+_B} &&	{\Sigma^-_B}
                                          && {\vec{\Sigma}^-_C \cap \cat P_C}
                                                   &&  {\vec{\Sigma}^{-}_C \setminus \cat P_C}   \\
      &\ar@{.}[r]& \ar@{-}[r] && && &\ar@{-}[r]&\ar@{.}[r]&&  \\
      && && &{\cat G} & && &&
      \save "2,4"+<-3pt,8pt>."4,8"*[F-:<10pt>]\frm{}\restore
   }\]

\begin{example}
   Let $B$ be the domestic quasitube algebra given by the fully commutative quiver
   \[\xymatrix@R=6pt{
         &	2\ar[dl]	&	4\ar[ddl]\ar[l]	&	\\
       1 &		&		& 6\ar[ul]\ar[dl]\\
         &	3\ar[ul]	&	5\ar[uul]\ar[l]	&	}\]
   and $C$ be the hereditary algebra given by the quiver
   \[\xymatrix@R=6pt{
      	4 	&						&	7\ar[dl]\\
            &	6\ar[dl]\ar[ul]&\\
        5   &						&	8\ar[ul].	}\]
   Then the algebra $A$ given by the quiver
   \[\xymatrix@R=6pt{
        &	2\ar[dl]	&	4\ar[ddl]\ar[l]	&						& 7\ar[dl]	\\
      1 &				&							& 6\ar[ul]\ar[dl]	&				\\
        &	3\ar[ul]	&	5\ar[uul]\ar[l]	&						& 8\ar[ul]  }\]
   bound by $\rad^4A=0$, two zero-relations of length three from 7 to 3 and from 8 to 2, respectively,
   and all possible commutativity relations is a finite gluing of $B$ and $C$.
   We draw below the central part of the glued component of $\Gamma(\mod A)$
   \[   \xymatrix@R=10pt{
               &			      &		&						& {\sm{7\\6\\4\,5\\2}}
                                                               \ar[dr]		&\\
               &{\sm{4\,5\\3}}
                     \ar[ddr]	&		& {\sm{6\\4\,5\\2}}
                                          \ar[ddr]
                                           \ar[ur]		&							&{\sm{7\\6\\4\,5}}
                                                                                  \ar[ddr]&		& {\sm{8\\6\,6\\4\,5}}\ar[ddr]      \\
               &{\sm{6\\4\,4\,5\\2\,3}}
                  \ar[dr]		&		& 5\ar[dr]			&							&	{\sm{6\\4}}
                                                                                 \ar[dr]	&		& {\sm{7\,8\\6\,6\\4\,5\,5}}
                                                                                                   \ar[dr]  \\
     \cdots\qquad
        \ar[uur]\ar[ur]
        \ar[ddr]\ar[dr]
               &					&{\sm{6\\4\,4\,5\,5\\2\,3}}
                                        \ar[uur]
                                        \ar[ur]
                                        \ar[ddr]
                                        \ar[dr]
                                    &                 & {\sm{6\\4\,5}}
                                                             \ar[uur]
                                                             \ar[ur]
                                                             \ar[ddr]
                                                             \ar[dr]			&					& {\sm{7\,8\\6\,6\,6\\4\,4\,5\,5}}
                                                                                                 \ar[uur]
                                                                                                 \ar[ur]
                                                                                                \ar[ddr]
                                                                                                 \ar[dr]
                                                                                                &							& \qquad\cdots      \\
               &{\sm{6\\4\,5\,5\\2\,3}}
                  \ar[ur]		&		& 4\ar[ur]			&							&	{\sm{6\\5}}
                                                                                 \ar[ur] &		& {\sm{7\,8\\6\,6\\4\,4\,5}}
                                                                                                     \ar[ur]  \\
               &{\sm{4\,5\\2}}\save[]+<2pc,-1.8pc>*{\Sigma^+_B}\restore
                  \ar[uur]		&		& {\sm{6\\4\,5\\3}}
                                          \ar[uur]
                                           \ar[dr]		&							&{\sm{8\\6\\4\,5}}
                                                                              \ar[uur]	  &		& {\sm{7\\6\,6\\4\,5}}\ar[uur]
                                                                                                   \save[]+<-2pc,-2.2pc>*{\Sigma^-_C}\restore      \\
               &		&		&						& {\sm{8\\6\\4\,5\\3}}
                                                            \ar[ur]		   &					&		&
               \save "2,2"+<-20pt,10pt>.{"6,3"+<15pt,-10pt>}*[F-:<6pt>]\frm{}\restore
               \save "2,7"+<-20pt,15pt>.{"6,8"+<20pt,-15pt>}*[F-:<6pt>]\frm{}\restore
                                                                  }\]
\end{example}

\subsection{Representation dimension of gluings}
We now prove that a finite gluing of two algebras having representation dimension three
also has representation dimension three under a reasonable hypothesis.

\begin{proposition}
    Let $A = B*C$.
    We assume that $\repdim B = 3$, $\repdim C = 3$ and the slice module $\Sigma^-_C$
    is a direct summand of an Auslander generator for mod $C$. Then $\repdim A = 3$.
\end{proposition}
\begin{proof}
   We introduce some notation.
   Let $M$ denote an Auslander generator for $\mod B$, and $N$ denote an Auslander generator for $C$ having the slice module of $\Sigma^-_C$
   as a direct summand.
   Finally, let $L$ denote the direct sum of the (finitely many) indecomposable modules in the glued component of $\Gamma(\mod A)$
   which are successors of $\Sigma^+_B$ and predecessors of $\Sigma^-_C$.
   We claim that
      \[ \overline M_A = A\oplus \D A\oplus L\oplus M\oplus N\]
    is an Auslander generator for $\mod A$.

   Let indeed $X$ be an indecomposable $A$-module.
   Assume first that $X\in\ivec{\Sigma}^+_B$.
   In particular, $X$ is an indecomposable $B$-module and has a minimal $\add M_B$-approximating sequence
   \[\xymatrix{ 0\ar[r]&M_1\ar[r]&M_0\ar[r]&X\ar[r]&0}\]
   with $M_1$, $M_0\in\add M\subseteq \add \overline M$.
   We claim that this sequence is also an $\add \overline M$-approximating sequence in $\mod A$.
   Let $f:M'\to X$ be a nonzero morphism with $M'\in\add \overline M$ indecomposable.
   Because $f$ is nonzero, $M'\in\ivec{\Sigma}^+_B\setminus \cal Q_B$.
   Therefore $M'\in\add M$.
   But then $f$ lifts to a morphism $M'\to M_0$.

   Let $X\in \add L$, then there is nothing to show.

   Finally, let $X\in\vec{\Sigma}^-_C$.
   Then $X$ is a $C$-module and has a minimal $\add N_C$-approximating sequence
   \[\xymatrix{ 0\ar[r]&N_1\ar[r]&N_0\ar[r]&X\ar[r]&0}\]
   with $N_1$, $N_0\in\add N\subseteq \add \overline M$.
   We claim that this sequence is also an $\add \overline M$-approximating sequence.
   Let $f:M'\to X$ be a nonzero morphism with $M'$ indecomposable.
   If $M'$ is not a successor of  $\Sigma^-_C$, then the morphism $f$ lifts to a morphism $M'\to N_0$.
   If $M'$ is a successor of  $\Sigma^-_C$ in $\add \overline M$, then $M'$ is an indecomposable $C$-module.
   So $f$ lifts to a morphism $M'\to N_0$ because the above sequence is an $\add N$-approximating sequence.
\end{proof}

\subsection{Induction}

Let $A = B*C$ be a finite gluing of $B$ and $C$ as above. If $A$ has a unique preinjective component containing a left section $\Sigma^+_A$, then we can define in the same way a finite gluing of $A$ with an algebra $D$ having a unique postprojective component containing a right section. Inductively, assuming that $B_1, \cdots B_n$ are a finite sequence of algebras having a unique preinjective component containing a left section $\Sigma^+_{B_{i}}$ and a unique postprojective component containing a right section $\Sigma^-_{B_{i+1}}$, for $ 1 \leq i < n$,
then we say that $A = B_1* \cdots * B_n$ is a finite gluing of the $B_i$ if $A = (B_1* \cdots * B_{n-1})* B_n$.

\begin{corollary}
    Let $A = B_1* \cdots * B_n$ where we assume that $\repdim B_i = 3$
    and the slice modules of $\Sigma^+_{B_i}$, $\Sigma^-_{B_{i+1}}$ for $ 1 \leq i < n$  are direct summands
    of an Auslander generator for mod $B_{i+1}$. Then $\repdim A = 3$.
\end{corollary}
\begin{proof}
   This is done by induction on $n\ge 2$, the case $n=2$ being Proposition 4.2.
\end{proof}

\subsection{Duplicated and replicated algebras}
We end this section with an application of gluing to a class of algebras of finite global dimension, which are closely related to selfinjective algebras. Let $B$ be an algebra and consider the matrix algebra
$$\overline{B} = \left(  \begin{array}{cc}
                            B & 0 \\
                            \D B & B \\
                    \end{array} \right)$$
with the ordinary matrix addition and the multiplication induced from the bimodule structure of $\D B$.

This algebra is called the \emph{duplicated algebra} of $B$, see for instance~\cite{A2}.

\begin{proposition}
    Let $B$ be a tilted algebra of euclidean type.
    \begin{enumerate}[\indent(a)]
        \item If $B$ admits a complete slice in its postprojective component,
        then there exist domestic quasitube algebras $C_1, C_2$ such that $\overline{B} = B*C_1*C_2$.
        \item If $B$ admits a complete slice in its preinjective component,
        then there exist domestic quasitube algebras $C_1, C_2$ such that $\overline{B} = C_1*C_2*B$.
    \end{enumerate}
    In particular, $\repdim\overline{B} = 3$.
\end{proposition}
\begin{proof}
   We only prove (a), because the proof of (b) is similar.
   We use the description of $\Gamma(\mod \overline B)$ given in~\cite{A2}.
   Let $\Sigma$ be a complete slice in $\Gamma(\mod B)$. Then $\Sigma$ embeds as a section in the stable part of the Auslander-Reiten
   quiver of the trivial extension $T(B)=B\ltimes\D B$ of $B$
   by its minimal injective cogenerator bimodule.
   Then, an exact fundamental domain for $\Gamma(\mod T(B))$ inserts in between the predecessors and the successors of $\D B$
   in $\Gamma(\mod B)$ to yield $\Gamma(\mod \overline B)$. Thus $\Gamma(\mod \overline B)$ has the following shape
   \begin{center}
      \includegraphics[width=0.7\linewidth]{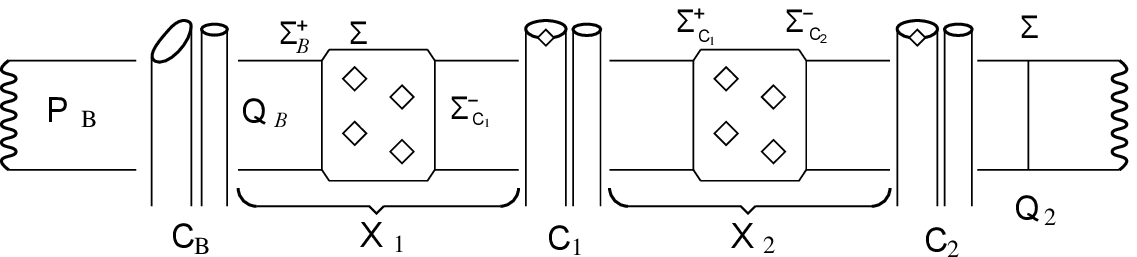}
   \end{center}
   where the diamonds represent possible occurrence of projective-injectives.
   It thus consists of:
   \begin{enumerate}[\indent(a)]
      \item a postprojective component $\cal P_B$ which is the postprojective component of $\Gamma(\mod B)$;
      \item a separating tubular family of (extended) tubes $\cat C_B$, which are the tubes in $\Gamma(\mod B)$;
      \item a transjective component $\cal X_1$  generally containing projective-injectives and having $\Sigma$ as left section.
      The predecessors of $\Sigma$ in $\cat X_1$ are exactly the predecessors of $\Sigma$ in the preinjective component of $\Gamma(\mod B)$;
      \item a separating family of quasitubes $\cat C_1$;
      \item a second transjective component $\cal X_2$ also generally containing projective-injectives;
      \item another separating family of quasitubes $\cat C_2$;
      \item a preinjective component $\cat Q_2$ having a left section isomorphic to $\Sigma$. The successors of this left section in $\cat Q_2$
      are exactly the successors of $\Sigma$ in the preinjective component of $\Gamma(\mod B)$.
   \end{enumerate}
   Thus we have $\ind\overline B=\cat P_B\vee\cat C_B\vee \cat X_1\vee\cat C_1\vee\cat X_2\vee \cat C_2\vee \cat Q_2$.

   Let $C_1$ be the support algebra of the family $\cat C_1$ of quasitubes inside $\overline B$.
   That is, $C_1$ is the full subcategory of $\overline B$ consisting of those $x\in\overline B_0$
   such that there exists a module $M$ in $\cat C_1$ satisfying $M(x)\ne 0$.
   Because every $C_1$-module is also a $\overline B$-module, $\cat C_1$ is a separating family of quasitubes in $\mod C_1$ and moreover
   $C_1$ is domestic.
   Therefore, $C_1$ is a domestic quasitube algebra.
   Similarly, the support algebra $C_2$ of the family $\cat C_2$ of quasitubes is a domestic quasitube algebra.

   We now show how to realise $\overline B$ as a finite gluing of $B$, $C_1$ and $C_2$.
   As observed, there exist a finite number of projective-injectives in the component $\cat X_1$.
   First, let $\Sigma_B^+$ denote the left section isomorphic to $\Sigma$ in $\cat X_1$.
   It precedes all the projective-injectives.
   Let next $\Sigma^-_{C_1}$ be any right section in $\cat X_1$ which succedes these projective-injectives.
   Similarly, let $\Sigma_{C_1}^+$ and $\Sigma_{C_2}^-$ be respectively a left section and a right section in $\cat X_2$,
   preceding and following the projective-injectives in that component.
   Looking at the definition of finite gluing now shows that $\overline B=B*C_1*C_2$.

   As for the last assertion, it follows directly from the main
   result of~\cite{APT}, theorem 3.4 and corollary 4.3, that $\repdim \overline B=3$.
\end{proof}

The preceding proposition can easily be generalised.
For $n\ge2$
    \[ B^{(n)} = \left( \begin{array}{ccccc}
                      	B_1 &  0  &        &   0     &	\\
                      	E_1 & B_2 &        &  		  &	\\
                      	    & E_2 & B_3    & 		  &	\\
                            &     & \ddots &	\ddots  &	\\
                      	 0  &     &        &  E_{n-1}&   B_n
                      \end{array}\right)   \]
be the lower triangular matrix algebra, where $B_i = B$ and $E_i = \D B$ for each $i$.
The addition is the usual addition of matrices and the multiplication induced from
the bimodule structure of $\D B$ and the morphisms $\D B \otimes_B \D B \To 0$.
This algebra is called the \emph{$n$-replicated algebra} of $B$.
It is easily shown, as in the proposition above, that, if $B$ is tilted of euclidean type,
then $\repdim B^{(n)} = 3$ for any $n$.

\section{Orbit Algebras of Repetitive Categories}
\subsection{Repetitive categories}

The selfinjective algebras of euclidean type are orbit algebras of repetitive categories. In this section, we recall the definitions and results on these algebras that are needed in the proof of our main theorem.

Let $B$ be a basic and connected algebra and $\set{e_1,\cdots,e_n}$ a complete set of primitive orthogonal idempotents for $B$.

Following~\cite{HW}, the \emph{repetitive category} $\hat B$ of $B$ is the category
having as objects $e_{m,i}$ with $(m,i) \in \Ent \times \set{1,\cdots,n}$
and where the morphism spaces are defined by
        \[\hat B(e_{r,i},e_{s,j}) = \begin{cases}
                                           e_jBe_i  & \text{if} s = r \\
                                         \D(e_iBe_j) & \text{if} s = r+1 \\
                                                0   & \text{otherwise.}
                                    \end{cases} \]

The repetitive category is a connected locally bounded selfinjective category.
A group $G$ of automorphisms of $\hat{B}$ is called \emph{admissible} if $G$ acts freely on
the objects of $\hat{B}$ and has finitely many orbits.

We define the Nakayama automorphism $\nu_{\hat B}$ of $\hat B$ to be the automorphism defined on the objects by
  \[\nu_{\hat B}(e_{m,i}) = e_{m,i+1}    \]
for every $(m,i)\in \Ent\times\set{1,\dotsc,n}$ and in the obvious way on the morphisms.
Then the infinite cyclic group $(\nu_{\hat B})$ generated by $\nu_{\hat B}$
is an admissible group of  automorphisms of $\hat B$.

Let now $\varphi$ be an automorphism of the category $\hat B$.
Thus $\varphi$ is said to be:
\begin{enumerate}[\indent(a)]
   \item \emph{positive} if, for each $(m,i)\in\Ent\times\set{1,\dotsc,n}$, we have $\varphi(e_{m,i})=e_{p,j}$ for some $p\ge m$
   and $j\in\set{1,2,\dotsc,n}$.
   \item \emph{rigid} if, for each $(m,i)\in\Ent\times\set{1,\dotsc,n}$, we have $\varphi(e_{m,j})=e_{m,j}$ for some  $j\in\set{1,2,\dotsc,n}$.
   \item \emph{strictly positive} if it is  positive but not rigid.
\end{enumerate}
For instance, $\nu_{\hat B}$ is a strictly positive automorphism of $\hat B$.

The following structure theorem for admissible groups of automorphisms is a consequence of the results of~\cite{BS1,BS2,S1}.

\begin{thm}
   Let $B$ be a tilted algebra of euclidean type and $G$ be a torsion-free admissible group of automorphisms of $\hat B$. Then $G$ is an infinite cyclic group generated by a strictly positive automorphism of one of the forms
   \begin{enumerate}[\indent(a)]
      \item $\sigma\nu^k_{\hat B}$ for a rigid automorphism $\sigma$ and some $k\ge 0$, or
      \item $\mu\varphi^{2k+1}$ for a rigid automorphism $\mu$, a strictly positive automorphism $\varphi$ such that $\varphi^2=\nu_{\hat B}$ and some $k\ge 0$.\qedhere
   \end{enumerate}
\end{thm}

Note that, if $B$ is tilted of type $\tilde{\mathbb{E}}$, then $\hat B$ does not admit
a strictly positive automorphism $\varphi$ such that $\varphi^2=\nu_{\hat B}$, see~\cite{LS}.
We refer to~\cite{BS1,BS2} for a complete description of the repetitive categories $\hat B$
of the tilted algebras $B$ of types $\tilde{\mathbb A}$ and $\tilde{\mathbb D}$ with $\varphi^2=\nu_{\hat B}$
for some strictly positive automorphism $\varphi$ of $\hat B$.

\subsection{Orbit categories and the Auslander-Reiten quiver of a repetitive category}

We now define orbit categories~\cite{G}.
Let $B$ be a basic and connected algebra and $G$ be an admissible group of automorphisms of $\hat B$.
The \emph{orbit category} $\hat B/G$ has as objects the $G$-orbits of objects of $\hat B$.
Given $a,b\in\left(\hat B/G\right)_0$, the morphism space $\hat B/G(a,b)$ is defined as
   \[\hat B/G(a,b)=\set{(f_{yx})\in\prod_{(x,y)\in a \times b} \hat B/G(x,y)~~ \mid~~ g\big(f_{yx}\big)=f_{g(x),g(y)},\;\text{for all $g\in G$, $x\in a$, $y\in b$}}.\]
In this situation, there exists a natural functor $F:\hat B\to\hat B/G$ called the associated \emph{Galois covering functor}, which assigns to any object $x\in\hat B_0$
its $G$-orbit $Gx$, and maps a morphism $\xi\in\hat B(x,y)$ to the family $F\xi$ such that
   \[ (F\xi)_{h(y),g(x)}=\begin{cases}
                           g(\xi)&\text{if $h=g$}\\
                           0&\text{if $h\ne g$}.
                        \end{cases}\]
Moreover, the functor $F$ induces $\kk$-linear isomorphisms
   \[  \bigoplus_{Fx=a}\hat B(x,y)\cong \hat B/G(a,Fy)\quad ;\quad \bigoplus_{Fy=b}\hat B(x,y)\cong \hat B/G(Fx,b)\]
   for all $x,y\in \hat B_0$; $a,b\in\left(\hat B/G\right)_0$.
Because $G$ is admissible, $\hat B/G$ is a category with finitely many objects and we may (and shall) identify it to the finite dimensional
algebra $\bigoplus\left(\hat B/G \right)$ which is the sum of all the morphism spaces in $\hat B/G$.

For instance, the orbit algebra of $\hat B$ by the (admissible) automorphism group $(\nu_{\hat B})$ generated by the Nakayama automorphism $\nu_{\hat B}$ is the
trivial extension $T(B)$ of $B$ by $\D B$.

Now we consider $\hat B$-modules.
We denote by $\mod \hat B$ the category of all the contravariant functors from $\hat B$ to $\mod \kk$, which we call finite dimensional $\hat B$-modules.
An admissible group $G$ of automorphisms of $\hat B$ also acts on $\mod\hat B$ by
   \[ {^g\!M}=Mg^{-1}\]
   for all $\hat B$-modules $M$ and all $g\in G$.

The Galois covering $F: \hat B \To \hat B / G$ induces
the so-called \emph{pushdown functor} $F_{\lambda}: \mod\hat B \To \mod\hat B/G$ such that
   \[ \left(F_\lambda M\right)(a)=\bigoplus_{x\in a}M(x)\]
for all $\hat B$-modules $M$ and all $a\in(\hat B/G)_0$, see~\cite{BG,G}.

Assume that the group $G$ is torsion-free.
Because of~\cite{G}, the functor $F_\lambda$ preserves almost split sequences and induces an embedding from the set of
$G$-orbits $\left(\ind\hat B\right)/G$ of isoclasses of indecomposable $\hat B$-modules into the set $\ind\left(\hat B/G\right)$ of isoclasses
of indecomposable $\hat B/G$-modules.

The density theorem proved in~\cite{DS1,DS2} says that $F_\lambda$ is dense whenever the category $\hat B$ is \emph{locally support-finite},
that is, for each $a\in\hat B_0$, the full subcategory $\hat B_a$ of $\hat B$, given by the supports of all $M$ in $\ind \hat B$ such that $M(a)\ne 0$, is finite.

If $F_\lambda$ is dense, then it induces an isomorphism between the orbit quiver $\Gamma(\mod \hat B)/G$ of $\Gamma(\mod \hat B)$ under the action of $G$ and the Auslander-Reiten quiver $\Gamma \left( {\mod\left( \hat B/G \right)} \right)$ of $\hat B/G$.

Moreover, we are able to describe the Auslander-Reiten quiver of the repetitive category of a tilted algebra of euclidean type~\cite{ANS}.

\begin{thm}
   Let $B$ be a tilted algebra of euclidean type $\vec\Delta$.
   Then the Auslander-Reiten quiver of  $\hat B$ is of the form
        $\Gamma(\mod{\hat B}) = \bigvee_{q\in\Ent}(\cal{X}_q \vee \cal{C}_q)$
  where for each $q \in \Ent$,
    \begin{enumerate}[\indent(a)]
        \item $\cal{X}_q$ is an acyclic component whose stable part is of the form $\Ent\vec{\Delta}$,
        \item $\cal{C}_q$ is a family $(\cal{C}_{q,\lambda})_{\lambda \in \mathbb{P}_1(\kk)}$ of quasitubes,
        \item $\nu_{\hat B}(\cal{X}_q) = \cal{X}_{q+2}$ and $\nu_{\hat B}(\cal{C}_q) = \cal{C}_{q+2}$,
        \item $\cal{X}_q$ separates             $\bigvee_{p < q }(\cal{X}_p \vee \cal{C}_p)~~ $
                            from $~~\cal{C}_q \vee \left(\bigvee_{p > q }(\cal{X}_p \vee \cal{C}_p)\right)$,
        \item $\cal{C}_q$ separates $\bigvee_{p < q}(\cal{X}_p \vee \cal{C}_p)  \vee \cal{X}_q~~$
                           from    $~~\bigvee_{p > q}(\cal{X}_p \vee \cal{C}_p)$.\qedhere
    \end{enumerate}
\end{thm}

The description of the previous theorem is said to be the \emph{canonical decomposition} of $\Gamma(\mod \hat B)$.

\subsection{Structure of the repetitive category}

We need one more concept from~\cite{HW}.
Let $B$ be a triangular algebra, then $\hat B$ is also triangular.
We identify $B$ with the full subcategory of $\hat B$ with object set $\set{e_{i,0}\mid i\in\set{1,\dotsc, n}}$.
Let now $i$ be a sink in $Q_B$.
The \emph{reflection} $S_i^+B$ of $B$ at $i$ is the full subcategory of $\hat B$ given by the objects $e_{0,j}$ with $j\in\set{1,\dotsc, n}\setminus\set{i}$ and
$e_{1,i}=\nu_{\hat B}(e_{0,i})$.
In this case, the quiver $\sigma_i^+ Q_B=Q_{S_i^+B}$ of $S_i^+B$ is called the \emph{reflection} of $Q_B$ at $i$.
Observe that $\hat B=\widehat{S_i^+B}$.

A \emph{reflection sequence of sinks} is a sequence $i_1,\dotsc,i_t$ of points in $Q_B$ such that, for each $s\in\set{1,\dotsc,t}$, the point $i_s$
is a sink in the quiver $\sigma_{i_{s-1}}^+\cdots \sigma_{i_1}^+Q_B$.

Finally, for a sink $i$ in $Q_B$, we denote by $T_i^+B$ the full subcategory of $\hat B$ having as objects those of $B$ and $e_{1,i}=\nu_{\hat B}(e_{0,i})$.
We note that $T_i^+B$ is isomorphic to the one-point extension $B[I(i)_B]$ of $B$ by the indecomposable injective $B$-module $I(i)$ at the point $i$.

We are now able to state the following result~\cite{ANS}.

\begin{thm}
   Let $B$ be a tilted algebra of euclidean type $\vec{\Delta}$ and let
      \[ \Gamma(\mod \hat B)=\bigvee_{q\in\Ent}(\cat X_q\vee \cat C_q)\]
    be the canonical decomposition of $\Gamma(\mod \hat B)$.
    For any $q\in\Ent$, we have:
    \begin{enumerate}[\indent(a)]
       \item The support algebra $B_q$ of $\cat{C}_q$ is a domestic quasitube algebra, which is a quasitube enlargement of a tame concealed full convex subcategory $C_q$ of $\hat B$.
       \item $\Gamma(\mod B_q)=\cat P^{B_q}\vee \cat C_q\vee\cat Q^{B_q}$, where $\cat P^{B_q}$ and $\cat Q^{B_q}$ are respectively a postprojective and a preinjective component, both of euclidean type, and $\cat C_q$ separates $\cat P^{B_q}$ from $\cat Q^{B_q}$.
       \item The support algebra $B_q^-$ of $\cat P^{B_q}$ is a domestic tubular coextension of $C_q$ and the support algebra $B_q^+$ of $\cat Q^{B_q}$ is a domestic tubular extension of $C_q$.
       \item There is a reflection sequence of sinks $i_1,\dotsc,i_r$ of $Q_{B_q^-}$ (possibly empty) such that $B_q^+=S_{i_r}^+\cdots S_{i_1}^+B_q^-$ and
             $B_q=T_{i_r}^+\cdots T_{i_1}^+B_q^-$.
       \item There is a reflection sequence of sinks $j_1,\dotsc,j_s$ of $Q_{B_q^+}$ (nonempty) such that $B_{q-1}^-=S_{j_s}^+\cdots S_{j_1}^+B_{q-1}^+$ and
       $D_q=T_{j_s}^+\cdots T_{j_1}^+B_{q-1}^+$ is the support algebra of $\cat X_q$. Hence, $\cat X_q$ contains at least one projective module.
       \item There is a cofinite full translation subquiver $\cat{X}_q^-=(-\Nat)\vec{\Delta}$ of $\cat P^{B_q}$ which is a full translation subquiver of $\cat X_q$ closed under successors.
       \item There is a cofinite full translation subquiver $\cat{X}_q^+=\Nat\vec{\Delta}$ of $\cat Q^{B_q}$ which is a full translation subquiver of $\cat X_{q+1}$ closed under predecessors.
      \end{enumerate}
      In particular, $\hat B$ is locally support-finite.  \qed
\end{thm}

\section{Selfinjective Algebras of Euclidean Type}

\begin{definition}
    A selfinjective algebra $A$ is said to be of \emph{euclidean type} if
    there exist a tilted algebra $B$ of euclidean type and
    an admissible infinite cyclic group $G$ of automorphisms of $\hat B$
    such that $ A = \hat B / G$.
\end{definition}

Examples of selfinjective algebras of euclidean type are provided by trivial extensions
of tilted algebras of euclidean type.

It has been proven in~\cite{S1} that
a selfinjective algebra is of euclidean type if and only if it is
representation-infinite, domestic
and admits a simply connected Galois covering (in the sense of~\cite{AS2}).

We are now able to prove the main result of the paper.

\begin{thm}
    Let $A$ be a selfinjective algebra of euclidean type. Then $\repdim A = 3$.
\end{thm}
\begin{proof}
   Because $A$ is representation-infinite, it suffices to prove that $\repdim A\le 3$.
   Let $B$ be a tilted algebra of euclidean type $\vec{\Delta}$ and $G$ be an infinite cyclic admissible group of automorphisms of $\hat B$
   such that $A=\hat B/G$. Because of theorem 5.1, $G$ is generated by a strictly positive automorphism $g$ of $\hat B$ and, because of theorem 5.2,
   $\Gamma(\mod \hat B)$ admits a canonical decomposition
    \[ \Gamma(\mod \hat B)=\bigvee_{q\in\Ent}(\cat X_q\vee \cat C_q).\]
   Furthermore, for each $q\in\Ent$, we have algebras $B_q^-$, $B_q$ and $B_q^+$ which satisfy the conditions of theorem 5.3.

   Because $G$ also acts on the translation quiver $\Gamma(\mod \hat B)$, there exists $m>0$ such that
   $g(\cal X_q)=\cal X_{q+m}$ and $g(\cal C_q)=\cal C_{q+m}$ for each $q\in\Ent$.
   Then it follows from the definitions of $B^-_q$, $B_q$, $B_q^+$ that we also have
      \[ g(B_q^-)=B_{q+m}^-,\quad g(B_q)=B_{q+m}\quad\text{and } g(B_q^+)=B_{q+m}^+\]
   for each $q\in\Ent$.

   Because of theorem 5.3, we may chose in $\cat P^{B_q}=\cat P^{B_q^-}$ an euclidean right section $\Sigma^-_q$
   of type $\vec{\Delta}$ such that the full translation subquiver $\cat X_q^-$ of $\cat P^{B_q}$ given by all successors of $\Sigma_q^-$ in $\cat P^{B_q}$
   consists of modules having nonzero restrictions to the tame concealed full convex subcategory $C_q$,  and is a full translation subquiver of $\cat X_q$ closed under successors.

   Similarly, we may chose in $\cat Q^{B_q}=\cat Q^{B_{q}^+}$ an euclidean left section $\Sigma^+_q$
   of type $\vec{\Delta}$ such that the full translation subquiver $\cat X_q^+$ of $\cat Q^{B_q}$ given by all predecessors of $\Sigma_q^+$ in $\cat Q^{B_q}$
   consists of modules having nonzero restrictions to $C_q$,  and is a full translation subquiver of $\cat X_{q+1}$ closed under predecessors.

   We may assume that $g\left(\Sigma_q^-\right)=\Sigma_{q+m}^-$ and $g\left(\Sigma_q^+\right)=\Sigma^+_{q+m}$.
   Consequently, $g(\cat X_q^-)=\cat X^-_{q+m}$ and $g(\cat X_q^-)=\cat X^+_{q+m}$ for each $q\in\Ent$.

   For a given $q\in\Ent$, denote by $\cat Y_q$ the finite translation subquiver of $\cat{X}_q$ consisting of all modules which are successors of $\Sigma_{q-1}^+$
   and predecessors of $\Sigma_q^-$.
   Observe that every projective module of $\cat{X}_q$ lies in $\cat Y_q$.
   Moreover, we have $g(\cat Y_q)=\cat Y_{q+m}$ for any $q$.

   Now, for each $q$, let $M_q$ denote the direct sum of all modules in $\cat Y_q$, all injective $B_q^-$-modules lying in $\cat C_q$ and all projective $\hat B$-modules lying in $\cat C_q$.
   Then, clearly ${^g}\!M_q=M_{q+m}$ for any $q\in\Ent$.

   Finally, we set $M=\bigoplus_{i=0}^{m-1}M_i$.

   Let $F_\lambda:\mod \hat B\to \mod A$ be the pushdown functor associated to the Galois covering $F:\hat B\to\hat B/G=A$.
   We shall prove that $F_\lambda(M)$ is an Auslander generator for $A$.

   First, note that $A$ is a direct summand of $F_\lambda(M)$. Indeed, any indecomposable projective $A$-module is of the form $F_\lambda(P)$, for some
   indecomposable projective $\hat B$-module $P$.
   By definition of $M$, there exists $r\in\Ent$ such that $P$ is a direct summand of ${^{g^r}}\!M$.
   But then $F_\lambda(P)$ is a direct summand of $F_\lambda({^{g^r}\!M})=F_\lambda (M)$.
   We now prove that $\gldim \End M\le 3$, which will complete the proof.

   Let $Z$ be an indecomposable $A$-module which is not a direct summand of $F_\lambda(M)$.
   Because the pushdown functor is dense, there exists $i$ such that $0\le i<m$ and an indecomposable module $X\in(\cat X_i^-\setminus \Sigma_i^-)\vee\cat C_i\vee(\cat X_i^+\setminus\Sigma_i^+)$ such that $Z=F_\lambda(X)$.
   Moreover, if $X\in\cat C_i$, then $X$ is neither a projective $B_i$-module, nor an injective $B_i^-$-module.
   Because of theorem 3.5, there exists an $\add M_i$-minimal approximating sequence
   \[ \xymatrix@1{0\ar[r]&U\ar[r]^u&V\ar[r]^v&X\ar[r]&0}\]
   in $\mod\hat B$. Applying the exact functor $F_\lambda$ yields an exact sequence
   \[ \xymatrix@1{0\ar[r]&F_\lambda(U)\ar[r]^{F_\lambda(u)}&F_\lambda(V)\ar[r]^{F_\lambda(v)}&F_\lambda(X)\ar[r]&0}\]
   with $F_\lambda(U)$, $F_\lambda(V)\in\add F_\lambda(M)$.
   We recall also that $F_\lambda(X)=Z$.
   We claim that $F_\lambda(v):F_\lambda(V)\to Z$ is an $\add F_\lambda(M)$-approximation.
   Let $h:F_\lambda(M)\to F_\lambda(X)=Z$ be a nonzero morphism.
   The pushdown functor $F_\lambda:\mod \hat B\to \mod A$ is a Galois covering of module categories.
   In particular, it induces a vector space isomorphism
      \[\Hom_A(F_\lambda(M),F_\lambda(X))\cong \bigoplus_{r\in\Ent}\Hom_{\hat B}({^{g^r}}\!M,X).\]
     Therefore, for each $r\in\Ent$, there exists a morphism $f_r:{^{g^r}}\!M\to X$, all but finitely
     many of the $f_r$ being zero, such that $h=\sum_{r\in\Ent}F_\lambda(f_r)$.

     We claim that, for any $r\ge 1$, we have $\Hom_{\hat B}({^{g^r}}\!M,X)=0$.
     Indeed, $X\in\cat X_i\vee\cat C_i\vee\cat X_{i+1}$ for some $i$ with $0\le i<m$.
     On the other hand, for $r\ge 1$, the module ${^{g^r}}\!M$ is a direct sum of modules lying in $\bigvee_{j=0}^{m-1}(\cat X_{j+mr}\vee\cat C_{j+mr})$.
     This establishes our claim.

     Let now $f_r:{^{g^r}}\!M\to X$ be a nonzero morphism in $\mod \hat B$ for some $r\ge 0$.
     Applying theorem 5.2, we conclude that $f_r\,$ factors through a module in $\add M_i$.
     Because $v$ is an $\add M_i$-approximation, there exists a morphism $w_t:{^{g^r}}\!M\to V$ in $\mod \hat B$ such that $f_r=vw_r$.
     But then $F_\lambda(f_r)=F_\lambda(v)F_\lambda(w_r)$ with $F_\lambda(w_r):F_\lambda(M)\to F_\lambda(V)$ because
     $F_\lambda({^{g^r}}\!M)=F_\lambda(M)$.
     Summing up, there exists a morphism $w:F_\lambda(M)\to F_\lambda(V)$ such that $h=F_\lambda(v)w$.
     This concludes the proof.
\end{proof}

\subsection{Example}
    Let $B$ be the algebra given by the quiver $Q$
    \[\xymatrix@R=1pc{
        8 \ar[rdd]_{\varrho}        & 7 \ar[dd]^{\delta}
                                    &                       & 6 \ar[ld]_{\gamma}
                                                                \ar[dd]^{\xi}
                                                                \ar[rdd]^{\eta}     \\
                                    &                       & 5 \ar[ld]_{\beta}      \\
                                    & 2 \ar[rd]^{\alpha}    & 				&3   & 4 \\
                                    &                       & 1                     }
    \]
    bound by the relations $\varrho \alpha = 0$ and $\delta \alpha = 0$.
    Then $B$ is a tilted algebra of euclidean type $\widetilde{\mathbb D}_7$, being the one-point coextension
    of the hereditary algebra $H = K \vec{\Delta}$ of type $\mathbb{D}_6$, given by the full subquiver
    $\vec{\Delta}$ of $Q$ formed by the points $2,3,4,5,6,7,8$,
    by the (uniserial) simple regular module
        \[   R = \begin{matrix} 6 \\ 5 \\2 \end{matrix}  \]
    lying on the mouth of the unique stable tube of $\Gamma_H$ of rank $4$.
    Then $1,2,3,4,5,6,7,8$ is a reflection sequence of sinks in $Q_B = Q$ such that
    \begin{itemize}
        \item $S_1^+B$ is the algebra given by the quiver $\sigma_1^+Q$ of the form
            \[ \xymatrix@R=1pc{
                                        &                           & 1' \ar[rd]^{\eps} \\
                    8 \ar[rdd]_{\varrho}& 7 \ar[dd]^{\delta}
                                        &   & 6 \ar[ld]_{\gamma}
                                                \ar[dd]^{\xi}
                                                \ar[rdd]^{\eta}     &                   \\
                                        &   & 5 \ar[ld]_{\beta}     &                   \\
                                        & 2 &							  & 3                     & 4 }
            \]
            bound by the relations  $\eps \xi = 0$ and $\eps \eta = 0$, which is a tilted algebra of type
            $\widetilde{\mathbb{D}}_7$, being the one-point extension  of $H$ by the simple regular module $R$;
        \item $S_4^+S_3^+S_2^+S_1^+B$ is the algebra given by
            the quiver $\sigma_4^+\sigma_3^+\sigma_2^+\sigma_1^+Q$ of the form
            \[ \xymatrix@R=1pc{
                    & 2' \ar[ldd]_{\theta}
                         \ar[dd]^{\lambda}
                         \ar[rd]^{\alpha'}
                                            &                   & 3' \ar[dd]^{\omega} & 4'\ar[ldd]^{\mu} \\
                    &                       & 1' \ar[rd]^{\eps} &                     &                  \\
                8   & 7                     &                   & 6 \ar[ld]_{\gamma}  &                  \\
                    &                       & 5                 &                     &                     }
            \]
            bound by the relations $\omega \gamma = 0$ and $\mu \gamma = 0$,
            which is a tilted algebra of type  $\widetilde{\mathbb{D}}_7$, and isomorphic to $B$;
        \item $S_5^+S_4^+S_3^+S_2^+S_1^+B$ is the algebra given by the
                quiver $\sigma_5^+\sigma_4^+\sigma_3^+\sigma_2^+\sigma_1^+Q$ of the form
            \[\xymatrix@R=1pc{
                    &                        & 5' \ar[ld]_{\beta'}  &  \\
	                & 2' \ar[ldd]_{\theta}
                	     \ar[dd]^{\lambda}
                	     \ar[rd]^{\alpha'}   &                      & 3' \ar[dd]^{\omega} & 4' \ar[ldd]^{\mu}\\
	                &                        & 1' \ar[rd]^{\eps}    &                                        \\
	            8   &   7                    &                      & 6                  &                   }
            \]
            bound by the relations $\beta' \theta = 0$ and $\beta' \lambda = 0$,
            which is a tilted algebra of type $\widetilde{\mathbb{D}}_7$, and isomorphic to $S_1^+B$;
        \item $S_8^+S_7^+S_6^+S_5^+S_4^+S_3^+S_2^+S_1^+B$ is the algebra given by the quiver
                $\sigma_8^+\sigma_7^+\sigma_6^+\sigma_5^+\sigma_4^+\sigma_3^+\sigma_2^+\sigma_1^+Q$  of the form
            \[ \xymatrix@R=1pc{
                8' \ar[rdd]_{\varrho'} & 7' \ar[dd]^{\delta'} &                     & 6' \ar[ld]_{\gamma'}
                                                                                         \ar[dd]^{\xi'}
                                                                                         \ar[rdd]^{\eta'} &    \\
	                                   &                      & 5' \ar[ld]_{\beta'} &                          \\
	                                   & 2' \ar[rd]^{\alpha'} &                     & 3'                  & 4' \\
	                                   &                      & 1'                  &                           }
            \]
            bound by the relations $\varrho' \alpha' = 0$ and $\delta' \alpha' = 0$, which is a tilted algebra
            of type $\widetilde{\mathbb{D}}_7$, and isomorphic to $B$.
    \end{itemize}

    The repetitive category $\widehat{B}$ of $B$ is given by the quiver
      \[ \xymatrix@R=.9pc@C=1pc{
            \save[] +<5.mm,0mm> *{\dotsd}
                    \ar@{-}[rd] \restore
                                    & \vdots\ar[d]      & \save[] +<-5.mm,0mm> *{\dotst}
                                                            \ar[ld] \restore
                                                                & \vdots \ar[d] & \save[] +<-10.mm,0mm> *{\dotsd}
                                                                                    \ar[d] \restore             \\
                                    & (m+2,2) \ar[ldd]_{\theta_{m+1}}
                                              \ar[dd]^{\lambda_{m+1}}
                                              \ar[rd]^{\alpha_{m+2}}
                                                        &       & (m+2,3) \ar[dd]^{\omega_{m+1}}
                                                                                & (m+2,4) \ar[ldd]^{\mu_{m+1}}  \\
                                    &                   & (m+2,1) \ar[rd]^{\eps_{m+1}}                          \\
            (m+1,8) \ar[rdd]_{\varrho_{m+1}}
                                    & (m+1,7) \ar[dd]^{\delta_{m+1}}
                                                        &       & (m+1,6) \ar[ld]_{\gamma_{m+1}}
                                                                          \ar[dd]^{\xi_{m+1}}
                                                                          \ar[rdd]^{\eta_{m+1}}                 \\
                                    &                   & (m+1,5) \ar[ld]_{\beta_{m+1}}                         \\
                                    & (m+1,2) \ar[ldd]_{\theta_m}
                                              \ar[dd]^{\lambda_m}
                                              \ar[rd]^{\alpha_{m+1}}
                                                        &       & (m+1,3) \ar[dd]^{\omega_m}
                                                                                & (m+1,4) \ar[ldd]^{\mu_m}      \\
                                    &                   & (m+1,1) \ar[rd]^{\eps_m}                              \\
            (m,8) \ar[rdd]_{\varrho_m}
                                    & (m,7) \ar[dd]^{\delta_m}
                                                        &       & (m,6) \ar[ld]_{\gamma_m}
                                                                        \ar[dd]^{\xi_m}
                                                                        \ar[rdd]^{\eta_m}                       \\
                                    &                   & (m,5) \ar[ld]_{\beta_m}                               \\
                                    & (m,2) \ar[ldd]_{\theta_{m-1}}
                                            \ar[dd]^{\lambda_{m-1}}
                                            \ar[rd]^{\alpha_m}
                                                        &       & (m,3) \ar[dd]^{\omega_{m-1}}
                                                                                & (m,4) \ar[ldd]^{\mu_{m-1}}    \\
                                    &                   & (m,1) \ar[rd]^{\eps_{m-1}}                            \\
            (m-1,8)                 & (m-1,7)           &       & (m-1,6)                                       \\
            \save[] +<+10.mm,0mm> *{\dotsd}
                \ar@{-}[u] \restore & \vdots \ar@{-}[u] & \save[] +<5.mm,0mm>
                                                            *{\rule[20pt]{0pt}{10pt}\dotst\rule[20pt]{0pt}{10pt}}
                                                            \ar@{-}[ur] \restore
                                                                & \vdots \ar@{-}[u]
                                                                                & \save[] +<-5.mm,0mm> *{\dotsd}
                                                                                    \ar@{-}[ul] \restore        }
        \]
    bound by the relations
    \begin{gather*}
    	\theta_m \varrho_m = \lambda_m \delta_m = \alpha_{m+1} \eps_m \gamma_m \beta_m,
      \gamma_{m+1} \beta_{m+1} \alpha_{m+1} \eps_m = \xi_{m+1} \omega_m = \eta_{m+1} \mu_m, \,\,
      \varrho_m \alpha_m = 0,   \,\,    \varrho_m \lambda_{m-1} = 0,\,\,    \delta_m \alpha_m = 0,               \\
    	\delta_m \theta_{m-1} = 0,\,\,    \beta_m \theta_{m-1} = 0,   \,\,    \beta_m \lambda_{m-1} = 0,  \,\,
      \eps_m \xi_m = 0,  \,\,    \eps_m \eta_m = 0,   \,\,    \omega_m \gamma_m = 0,     \,\,
      \omega_m \eta_m = 0,      \,\,    \mu_m \gamma_m = 0,         \,\,    \mu_m \xi_m = 0,
    \end{gather*}
    for all $m \in \mathbb{Z}$.
    We identify the algebra $B$ with the full subcategory of $\widehat{B}$ given by the objects $(0,i)$,
    $i \in \set{1,2,3,4,5,6,7,8}$.

    Let $\varphi : \widehat{B} \to \widehat{B}$ be the automorphism of the category $\widehat{B}$ given by
    \begin{align*}
    	\varphi\big((m,1)\big) & = (m,5),     & \varphi\big((m,2)\big) & = (m,6),      & \varphi\big((m,3)\big) & =(m,7),        & \varphi\big((m,4)\big) & = (m,8),         \\
    	\varphi\big((m,5)\big) & =(m+1,1),    & \varphi\big((m,6)\big) & = (m+1,2),    & \varphi\big((m,7)\big) & =(m+1,3),      & \varphi\big((m,8)\big) & = (m+1,4),       \\
    	\varphi(\alpha_m)      & = \gamma_m,  & \varphi(\beta_m)       & = \eps_m,     & \varphi(\gamma_m)      & =\alpha_{m+1}, & \varphi(\xi_m)         & = \lambda_m,     \\
    	\varphi(\eta_m)        & = \theta_m,  & \varphi(\delta_m)      & = \omega_m,   & \varphi(\varrho_m)     & = \mu_m,       & \varphi(\lambda_m)     & = \xi_{m+1},     \\
    	\varphi(\theta_m)      & = \eta_{m+1} & \varphi(\eps_m)        & =\beta_{m+1}, & \varphi(\omega_m)      & =\delta_{m+1}, & \varphi(\mu_m)         & = \varrho_{m+1},
    \end{align*}
    for all $m \in \mathbb{Z}$.
    Observe that $\varphi^2 = \nu_{\widehat{B}}$.

    The orbit algebra $A_1 = \widehat{B}/(\varphi)$ is the algebra given by the quiver
    \[\xymatrix@R=1pc{
            && 3 \ar@<.5ex>^{\delta}[ld]\\
            5 \ar@<.5ex>^{\beta}[r]
              & 2 \ar@<.5ex>^{\alpha}[l] \ar@<.5ex>^{\xi}[ru] \ar@<.5ex>^{\eta}[rd] \\
            && 4 \ar@<.5ex>^{\varrho}[lu]}
    \]
    bound by the relations
    \[
      \alpha \beta \alpha \beta = \xi \delta = \eta \varrho, \,\,
      \varrho \alpha = 0,    \,\,     \delta \alpha = 0,     \,\,     \beta \xi = 0,     \,\,
      \beta \eta = 0,        \,\,     \delta \eta = 0,       \,\,     \varrho \xi = 0.      \]
    We note that $A_1$ is a symmetric one-parametric algebra.

    The orbit algebra
    $A_2 = \widehat{B}/(\varphi^2) = \widehat{B}/(\nu_{\widehat{B}}) = \operatorname{T}(B)$
    is the algebra given by the quiver
    \[
        \xymatrix@R=1pc{
            7 \ar@<.5ex>^{\delta}[rd] && 1 \ar^{\eps}[rd] && 3 \ar@<.5ex>^{\omega}[ld]\\
            & 2 \ar^{\alpha}[ru] \ar@<.5ex>^{\lambda}[lu] \ar@<.5ex>^{\theta}[ld]
              && 6 \ar^{\gamma}[ld] \ar@<.5ex>^{\xi}[ru] \ar@<.5ex>^{\eta}[rd] \\
            8 \ar@<.5ex>^{\varrho}[ru] && 5 \ar^{\beta}[lu] && 4 \ar@<.5ex>^{\mu}[lu]
        }
    \]
    bound by the relations
    \begin{gather*}
      \theta \varrho = \lambda \delta = \alpha \eps \gamma \beta,     \,\,
      \xi \omega = \eta \mu = \gamma \beta \alpha \eps,   \,\,
      \varrho \alpha = 0,    \,\,  \varrho \lambda = 0, \,\,    \delta \alpha = 0,  \,\,    \delta \theta = 0,  \\
      \beta \theta = 0,      \,\,   \beta \lambda = 0,  \,\,    \eps \xi = 0, \,\,    \eps \eta = 0, \,\,
      \omega \gamma = 0,     \,\,     \omega \eta = 0,   \,\,     \mu \gamma = 0,    \,\,     \mu \xi = 0.
    \end{gather*}
    We note that $A_2 = \operatorname{T}(B)$ is a $2$-parametric symmetric algebra.

    The orbit algebra   $A_3 = \widehat{B}/(\varphi^3)$ is the algebra given by the quiver
    \[\begin{xy}
            0;/r.3pc/:
            (0,0)*+{6}="6" ;
            (0,17.32)*+{5'}="5'" ;
            (-10,17.32)*+{2'}="2'" ;
            (10,17.32)*+{2}="2" ;
            (0,25.98)*+{7}="7" ;
            (0,34.64)*+{8}="8" ;
            (-5,8.66)*+{1'}="1'" ;
            (5,8.66)*+{5}="5" ;
            (-12.5,4.33)*+{3}="3" ;
            (12.5,4.33)*+{3'}="3'" ;
            (-20,0)*+{4}="4" ;
            (20,0)*+{4'}="4'" ;
            \ar @{->}^(.65){\gamma} "6";"5"
            \ar @{->}^(.35){\beta} "5";"2"
            \ar @{->}^(.65){\alpha} "2";"5'"
            \ar @{->}^(.35){\beta'} "5'";"2'"
            \ar @{->}^(.65){\alpha'} "2'";"1'"
            \ar @{->}^(.35){\eps} "1'";"6"
            \ar @/^/@{->}_{\xi} "6";"3"
            \ar @/^/@{->}_{\delta'} "3";"2'"
            \ar @/^/@{->}_{\lambda} "2'";"7"
            \ar @/^/@{->}_{\delta} "7";"2"
            \ar @/^/@{->}_{\xi'} "2";"3'"
            \ar @/^/@{->}_{\omega} "3'";"6"
            \ar @/^/@{->}^{\eta} "6";"4"
            \ar @/^/@{->}^{\varrho'} "4";"2'"
            \ar @/^/@{->}^{\theta} "2'";"8"
            \ar @/^/@{->}^{\varrho} "8";"2"
            \ar @/^/@{->}^{\eta'} "2";"4'"
            \ar @/^/@{->}^{\mu} "4'";"6"
      \end{xy} \]
    bound by the relations
    \begin{gather*}
      \theta\varrho = \lambda \delta = \alpha' \eps \gamma \beta,    \quad
      \alpha\beta' \alpha' \eps = \xi' \omega = \eta' \mu,           \quad
      \gamma\beta \alpha \beta' = \xi \delta' = \eta \delta',                                               \\
      \varrho\alpha = 0,    \quad \delta \alpha = 0,\quad \varrho \xi' = 0,     \quad \delta \eta' = 0,
                            \quad \beta\xi' = 0,    \quad \beta \eta' = 0,                                  \\
      \eps\xi = 0,          \quad \eps \eta = 0,    \quad \omega \gamma = 0,    \quad \omega \eta = 0,
                            \quad \mu\gamma = 0,    \quad \mu \xi = 0,          \quad \varrho' \alpha' = 0, \\
      \delta'\alpha'= 0,    \quad \varrho'\lambda = 0,  \quad \delta' \theta = 0,   \quad \beta' \theta = 0,
                            \quad \beta' \lambda = 0.
    \end{gather*}

    We note that $A_3$ is a $3$-parametric selfinjective algebra of euclidean type, which is not weakly symmetric.

    More generally, for any positive integer $r \geq 3$, the orbit algebra $A_r = \widehat{B}/(\varphi^r)$ is
    an $r$-parametric selfinjective algebra of euclidean type, which is not weakly symmetric.

    We mention that every domestic quasitube convex subcategory of $\widehat{B}$ is the algebra
    $\varphi^r(T_1^+ B)$, for some $r \in \mathbb{Z}$, and $T_1^+ B$ is given by the quiver
    \[\xymatrix@R=1pc{
                                 &                      & 1' \ar[rd]^{\eps}                         \\
            8 \ar[rdd]_{\varrho} & 7 \ar[dd]^{\delta}   &                   & 6 \ar[ld]_{\gamma}
                                                                                \ar[dd]^{\xi}
                                                                                \ar[rdd]^{\eta}     \\
                                 &                      & 5 \ar[ld]_{\beta}                         \\
                                 & 2 \ar[rd]^{\alpha}   &                   & 3                 & 4 \\
                                 &                      & 1                                         }
    \]
    bound by the relations $\varrho \alpha = 0$, $\delta \alpha = 0$, $\eps \xi = 0$
    and $\eps \eta = 0$.
    The algebras $\varphi^r(T_1^+ B)$, $r \in \Ent$, are the support algebras of the families
    $\mathscr{C}_q = (\mathscr{C}_{q,\lambda})_{\lambda \in \mathbb{P}_1(K)}$, $q \in \mathbb{Z}$, of quasitubes of
    $\Gamma(\mod \widehat{B})$, described in theorem~5.2.

    We also note that the support algebras of the acyclic components $\mathscr{X}_q$, $q \in \mathbb{Z}$,
    of $\Gamma(\mod \widehat{B})$, described in theorem~5.2,
    are of the forms $\varphi^r(T_4^+T_3^+T_2^+S_1^+ B)$, $r \in \mathbb{Z}$,
    where $T_4^+T_3^+T_2^+S_1^+ B$ is given by the quiver
    \[ \xymatrix@R=1pc{
                    & 2' \ar[ldd]_{\theta}
                         \ar[dd]^{\lambda}
                         \ar[rd]^{\alpha'}  &                   & 3' \ar[dd]^{\omega}   & 4' \ar[ldd]^{\mu} \\
                    &                       & 1' \ar[rd]^{\eps}                                             \\
        8 \ar[rdd]_{\varrho}
                    & 7 \ar[dd]^{\delta}    &                   & 6 \ar[ld]_{\gamma}
                                                                    \ar[dd]^{\xi}
                                                                    \ar[rdd]^{\eta}                         \\
                    &                       & 5 \ar[ld]_{\beta}                                             \\
                    & 2                     &                   & 3                     & 4                 }
    \]
    bound by the relations $\theta\varrho = \lambda\delta = \alpha'\eps\gamma\beta$,
    $\eps\xi = 0$, $\eps\eta = 0$, $\omega\gamma = 0$, $\omega\eta = 0$, $\mu\gamma = 0$ and $\mu\xi = 0$.
    Moreover, $T_4^+T_3^+T_2^+S_1^+ B$ is a gluing of the tilted algebras $S_1^+ B$ and $S_4^+S_3^+S_2^+S_1^+ B$.

\subsection*{Acknowledgement}
The first author was partially supported by a research fellowship within the project "Enhancing Educational Potential of Nicolaus Copernicus University
in the Disciplines of Mathematical and Natural Sciences" (project no. POKL.04.01.01-00-081/10).
He was also partially supported by the NSERC of Canada, the FRQ-NT of Québec and the Université de Sherbrooke.

The second and third authors were supported by the research grant DEC-2011/02/A/ST1/00216 of the Polish National Science Center.

Research on this paper was carried out while the first and the third author were visiting the second.
They are grateful to him and to the whole representation theory group in Toru\'n for their kind hospitality during their stay.

\section*{References}

\end{document}